\documentclass[final,1p,times]{elsarticle}

\usepackage{amssymb}
\usepackage{amsthm}

\usepackage{latexsym}
\usepackage{amsmath}
\usepackage{color}
\usepackage{graphicx}
\usepackage{indentfirst}
\usepackage{mathrsfs}
\usepackage{pdfsync}
\usepackage{hyperref}

\hypersetup{hypertex=true,
            colorlinks=true,
            linkcolor=blue,
            anchorcolor=blue,
            citecolor=blue}
\allowdisplaybreaks
\usepackage[utf8]{inputenc}
\def\e{\varepsilon}

\newtheorem{theorem}{\color{black}\indent Theorem}
\newtheorem{lemma}{\color{black}\indent Lemma}[section]
\newtheorem{proposition}{\color{black}\indent Proposition}

\newtheorem{remark}{\color{black}\indent Remark}[section]
\newtheorem{corollary}{\color{black}\indent Corollary}[section]

\DeclareMathOperator{\diag}{{diag}}

\journal{a}

\begin{document}

\begin{frontmatter}



\title{An Infinite-dimensional KAM Theorem with Normal Degeneracy}


\author{{Jiayin Du$^{a}$ \footnote{ E-mail address : dujy20@mails.jlu.edu.cn} ,~ Lu Xu$^{a,*}$ \footnote{*Corresponding author's e-mail address : xulu@jlu.edu.cn},~ Yong Li$^{a,b}$ \footnote{E-mail address : liyong@jlu.edu.cn} }\\
{$^{a}$College of Mathematics, Jilin University,} {Changchun 130012, P. R. China.}\\
{$^{b}$School of Mathematics and Statistics, and Center for Mathematics and Interdisciplinary Sciences, Northeast Normal
University, Changchun 130024, P. R.  China.}
 }

\begin{abstract}
In this paper, we consider a classical Hamiltonian normal form with degeneracy in normal direction. In previous results, one
needs to assume that the perturbation satisfies certain non-degenerate conditions in order to remove the degeneracy in the normal form. In stead
of that, we introduce a topological degree condition and a weak convexity condition, which are easy to be verified, and we prove the persistence of lower dimensional tori without any restriction on perturbation but only smallness and analyticity.
\end{abstract}






\end{frontmatter}



\tableofcontents
\newpage

\section{Introduction and Main Theorem}
\setcounter{equation}{0}
\setcounter{definition}{0}
\setcounter{proposition}{0}
\setcounter{lemma}{0}
\setcounter{remark}{0}
\subsection{Introduction}\label{sec:1}
In the present paper, we consider the following Melnikov's persistence for infinite dimensional Hamiltonian with degeneracy in certain normal directions,
which is described in the following form
\begin{equation}\label{eq1}
H(x,y,u,\bar{u},\xi)=N(y,u,\bar{u},\xi)+\varepsilon P(x,y,u,\bar{u},\xi,\varepsilon),
\end{equation}
where $(x,y,u,\bar{u})\in\mathcal{P}^{\textsf{a},\textsf{p}}=\mathbb{T}^n\times{\mathbb{C}^n}\times \ell^{\textsf{a},\textsf{p}}\times \ell^{\textsf{a},\textsf{p}},$ $\xi\in\Pi\subset\mathbb{R}^n$, $n\ge1$ is a given integer, $\mathbb{T}^{n}$ is the standard $n$-torus, $\textsf{a}\geq0$ and $\textsf{p}\geq1$ are given
constants and $\ell^{\textsf{a},\textsf{p}}$ is the Hilbert space, $\varepsilon>0$ is a small parameter. The Hamiltonian is real analytic in $(x,y,u,\bar{u})$ and Lipschitz in parameter $\xi$. More specifically, integrable part $N$ is in the following form
\begin{align*}
N=\langle\omega(\xi), y\rangle+\langle u,\Omega^*(\xi)\bar{u}\rangle+g(z,\xi),
\end{align*}
where  $\Omega^*(\xi)=\diag(\Omega_0(\xi),\Omega(\xi))$, $g(z,\xi)$ consists of high order terms which will be specified later.
 For given integer $b>0$, denote
$\mathcal{N}=\{j_1<j_2<\cdots<j_b\}\subset\mathbb{N}_+$, and
$$\Omega_0(\xi)=\diag(\Omega_0^{j_1},\cdots,\Omega_0^{j_b})\in\mathbb{R}^{b\times b},~~\Omega(\xi)=\diag(\Omega^j\in\mathbb{R}:j\in\mathbb{N}_+\backslash\mathcal{N}),$$
with $\Omega_0^{j_i}\equiv0,~~i=1,\cdots,b.$ Divided the vectors $(u,\bar{u})$ into the form of
$u=:(w_0,w)$, $\bar{u}=:(\bar{w}_0,\bar{w})$ with
\begin{eqnarray*}
&&w_0=(w_{j_i}\in\mathbb{C},i=1,\cdots,b),~~w=(w_j\in\mathbb{C},j\in\mathbb{N}_+\backslash\mathcal{N}),\\
&&\bar{w}_0=(\bar{w}_{j_i}\in\mathbb{C},i=1,\cdots,b),~~\bar{w}=(\bar{w}_j\in\mathbb{C},j\in\mathbb{N}_+\backslash\mathcal{N}).
\end{eqnarray*}
Denote $z=:(w_0,\bar{w}_0)$ and $g(z,\xi)=o(\|z\|_{\textsf{a},\textsf{p}}^2).$
Associated with standard  symplectic structure ${\rm d}y\wedge {\rm d}x+\sqrt{-1}{\rm d}\bar{u}\wedge {\rm d}u$, $\mathcal{T}_0^n=\mathbb{T}^n\times\{y=0\}\times\{u=0\}\times\{\bar{u}=0\}$ is a $n$-dimensional invariant rotational torus for
the unperturbed Hamiltonian. Studying the persistence of $\mathcal{T}_0^n$ under small perturbation is a classical problem,
but it becomes more challenging when the normal form is degenerate.

In the $1980$s, Kuksin \cite{kuksin1} and Wayne \cite{wayne} initially applied the KAM iterative process to construct the lower dimensional tori in the infinite dimensional Hamiltonian systems and proved the existence of quasi-periodic solutions for the nonlinear Schr\"{o}dinger equations and wave equations. Similar results also proved by P\"{o}schel in \cite{poschel1,poschel2}. Bourgain \cite{bourgain1,bourgain2,bourgain3}  obtained a sharp persistence result in both finite and infinite dimensional cases only under the first Melnikov conditions, which therefore allows
the multiplicity of normal frequencies.

However, when degeneracy occurs in normal directions, lower dimensional tori do not necessarily survive under small perturbation, even for
the Hamiltonian in finite dimensional. As a result, one needs to attach certain non-degenerate assumption on perturbation $P$.
For instance, the authors in \cite{han}  considered the persistence of lower dimensional for the Hamiltonian as follow
\begin{eqnarray} \label{hlyh}
H(I,\theta,z,\e)=\langle\omega,I\rangle+\langle A(\omega)z,z\rangle+h(z)+\e P(I,\theta,z,\omega,\e),
\end{eqnarray}
where $(I,\theta,z)\in{\mathbb T}^{d}\times{\mathbb R}^{d}\times{\mathbb R}^{2n}$, $A(\omega)$ admits zero eigenvalues
and $h(z)=O(|z|^3)$. Under the assumption
\begin{eqnarray}\label{hlassumption2}
\frac{\partial [P(\cdot,0,\omega,\e)]}{\partial z}=0,~\quad\quad~\det\frac{\partial^2 [P(\cdot,0,\omega,\e)]}{\partial z^2}\ne 0,
\end{eqnarray}
it was proved that most of the $d$-tori $T_{\omega}=\{\omega\}\times\{I=0\}\times\{z=0\}$ persist under small perturbation.
In fact, assumption (\ref{hlassumption2}) on the perturbation can remove the singularity of $A(\omega)$, hence, it yields the persistence of the majority of invariant tori under a suitable weak Melnikov non-resonance condition. Similar assumptions can be found in \cite{li} and \cite{xly}. Recently, Wu and Yuan \cite{wu} investigated the existence of KAM tori for the infinite dimensional Hamiltonian systems with finite number of zeros among normal frequencies. The authors introduced a vector which consist of the first order term of perturbation $P$  yielding at each iterative step, and they proved that there exists a KAM torus if
the sum of those vectors tends to zero. Here, our assumptions provide sufficient conditions to ensure the existence of normal equilibrium during KAM process. More precisely, the topological degree condition in \textbf{(A0)} is one of transversality and the weak convex condition in \textbf{(A0)} controls the size of the perturbed normal equilibrium.


 In fact, Hamiltonian (\ref{eq1}) can be seen as an integral function of Hamiltonian lattice equations with
 degeneracy in certain directions, studying the existence of lower dimensional tori is contribute to describe
 the dynamical stability for such equations. However, the perturbations in certain applications are of
 complicated forms, so that it is difficult to verify assumptions in \cite{han,wu,xly}.
 Motivated by that, we construct the topological degree condition as well as the weak convexity condition on $g(z,\xi)$,
  which allows to degeneracy in normal direction, meanwhile, it is easily to be verified. See \textbf{(A0)} in Section \ref{sec:result} for details.
 It should be pointed out that the topological degree can be used to study frequency-preserving in the KAM theory, see \cite{du,tong,xu2010}.
 The weak convexity condition in \textbf{(A0)} is also indispensable, we show a counter example in smooth case in the last section.

Beside the small divisor problem, the other difficulty during KAM iterative process is to eliminate the first order terms
of the perturbation averaged with respect to angle variable $x$, that is,  we need to solve an
average equation in $(u,\bar{u})$-direction. Assumption \textbf{(A0)} guarantees that the average equation
is solved by an equilibrium in normal direction.
Consequently, we prove Hamiltonian (\ref{eq1}) admits a family of lower dimensional tori parameterized by $\xi$ varying in
certain Cantor set without any restriction for the perturbation $P$ except for smallness and analyticity.

\subsection{Notations}

In order to state our main result, we need some notations.
\begin{itemize}
\item[{(i)}] {Use $|\cdot|$ to denote the supremum norm of vectors, its induced matrix norm, absolute value of functions, and Lebesgue measure for sets.}
\item[{(ii)}] Endow the Hilbert space $\ell^{\textsf{a},\textsf{p}}$ with the following norm
    \begin{align*}
||u||_{\textsf{a},\textsf{p}}^2=||{w}_0||_{\textsf{a},\textsf{p}}^2+||{w}||_{\textsf{a},\textsf{p}}^2<\infty,~~~~u=({w}_0,w)\in\ell^{\textsf{a},\textsf{p}},
\end{align*}
where $||{w}_0||_{\textsf{a},\textsf{p}}^2=\sum_{j_i\in\mathcal{N}}|w_{j_i}|^2{j_i}^{2\textsf{p}}{\rm e}^{2\textsf{a}j_i}$, $||{w}||_{\textsf{a},\textsf{p}}^2=\sum_{j\in\mathbb{N}_+\backslash\mathcal{N}}|w_j|^2j^{2\textsf{p}}{\rm e}^{2\textsf{a}j}$.
\item[{(iii)}] Denote the complex neighborhoods of $\mathcal{T}^n_0$
    \begin{align*}
D(s,r)=\{(x,y,u,\bar{u}):|\textrm{Im}x|<s,|y|<r^2,||z||_{\textsf{a},\textsf{p}}<r,||w||_{\textsf{a},\textsf{p}}<r^a,||\bar{w}||_{\textsf{a},\textsf{p}}<r^a\}
\end{align*}
with $0<s, r<1$, $a\geq 2$, which will be defined in (\ref{a}).
\item[{(iv)}] For $r>0$ and $\bar{\textsf{p}}\geq \textsf{p}$, we define the weighted phase space norms
\begin{align*}
\pmb{\|}W\pmb{\|}_r=\pmb{\|}W\pmb{\|}_{ \bar{\textsf{p}},r}=\frac{|X|}{r^{a-2}}+\frac{|Y|}{r^a}+\frac{||\tilde U||_{\textsf{a},\bar {\textsf{p}}}}{r^{a-1}}+||U_{\bar{w}}||_{\textsf{a},\bar {\textsf{p}}}+||V_{{w}}||_{\textsf{a},\bar {\textsf{p}}},
\end{align*}
for $W=(X,Y,U,V)\in\mathcal{P}^{\textsf{a},\bar{\textsf{p}}}$, $U=(U_{\bar w_0},U_{\bar w})$, $V=({V}_{w_0},{V}_{w})$, $\tilde U=(U_{\bar w_0},{V}_{w_0})$.
And we denote $$\pmb{\|}W\pmb{\|}_{r,D(s,r)}=\sup_{D(s,r)}\pmb\|W\pmb\|_r.$$
Assume that $W$ is Lipschitz in parameter $\xi$, we denote its Lipschitz semi-norm
$$\pmb\|W\pmb\|_r^{\mathcal{L}}=\sup_{\xi\neq\zeta,~ \xi,\zeta\in \Pi}\frac{\pmb\|\Delta_{\xi\zeta}W\pmb\|_r}{|\xi-\zeta|},$$
where $\Delta_{\xi\zeta}W=W(\cdot,\xi)-W(\cdot,\zeta)$. Moreover, the Lipschitz semi-norms of $\omega$ and $\Omega$, i.e., $|\omega|_{\Pi}^{\mathcal{L}}$ and $\pmb|\Omega\pmb|_{r,\Pi}^{\mathcal{L}}$,  are defined analogously to $\pmb\|W\pmb\|_r^{\mathcal{L}}$, where $\pmb|\Omega\pmb|_{r}=\sup_j|\Omega_j|j^{r}$.
\item[{(v)}] For $\lambda\geq0$, we define $\pmb\|\cdot\pmb\|_r^\lambda=\pmb\|\cdot\pmb\|_r+\lambda\pmb\|\cdot\pmb\|_r^{\mathcal{L}}$. Also, $\pmb\|\cdot\pmb\|_r^*$ stands for $\pmb\|\cdot\pmb\|_r$ or $\pmb\|\cdot\pmb\|_r^{\mathcal{L}}$.
\item[{(vi)}]  We introduce the notations $$\langle l\rangle_d=\max(1, |\sum_j j^dl_j|),~~~\mathcal{Z}=\{(k,l)\neq0, |l|\leq2\}\subset\mathbb{Z}^n\times\mathbb{Z}^\infty,~~~J=\left(\begin{array}{ccc}
0&I_{b}\\
-I_{b}&0\\
\end{array}\right),$$
where $I_{b}$ is the $b$-order identity matrix.
\end{itemize}

\subsection{Statement of results}\label{sec:result}
We consider (\ref{eq1}), i.e.,
\begin{equation*}
\left\{
\begin{array}{ll}
H:\mathbb{T}^n\times{G}\times \ell^{\textsf{a},\textsf{p}}\times \ell^{\textsf{a},\textsf{p}}\times \Pi\rightarrow \mathbb{R}^1,\\
H(x,y,u,\bar{u},\xi)=\langle\omega(\xi), y\rangle+\langle w, \Omega(\xi)\bar{w}\rangle+{g(z,\xi)}+\varepsilon P(x,y,u,\bar{u},\xi),
\end{array}
\right.
\end{equation*}
where $u=(w_0,w)$, $\bar u=(\bar w_0,\bar w)$, $z=(w_0,\bar w_0)$, $g=o(||z||_{\textsf{a},\textsf{p}}^2)$, $\omega$, $\Omega$, $g$ and $P$ are Lipschitz in parameter $\xi$, $g$ is real analytic in $z$, $P$ is real analytic in $(x,y,u,\bar{u})$ and $\varepsilon>0$ is a small parameter.

First, we make the following assumptions:
\begin{itemize}
\item[(\textbf{A0})]
{For fixed $\zeta_0=0\in O^o\subset \mathbb{C}^{2b}$, where $O$ is a bounded closed domain and {$O^o:=O\setminus\partial O$}, there are $\sigma>0$, $L\geq2$ such that
\begin{align*}
&\deg(\nabla{g(z)}-\nabla{g(\zeta_0)}, O^o, 0)\neq0,\\
&||\nabla{g(z)}-\nabla{g(z_*)}||_{\textsf{a},\bar {\textsf{p}}}\geq\sigma||z-z_*||_{\textsf{a},\textsf{p}}^L,~~z,z_*\in O,
\end{align*}
where $\bar p\geq p$, $\nabla g(z)=(\partial_{z_1}g(z), \partial_{z_2}g(z), \cdots, \partial_{z_{2b}}g(z))$.}
\item[(\textbf{A1})]
{The mapping $\xi\rightarrow\omega(\xi)$ is a lipeomorphism, that is, a homemorphism which is Lipschitz continuous in both directions. Moreover, for all integer vectors $(k,\ell)\in\mathbb{Z}^n\times\mathbb{Z}^\infty$ with $1\leq|\ell|\leq2$, $$|\{\xi:\langle k,\omega(\xi)\rangle+\langle\ell,\Omega(\xi)\rangle=0\}|=0,$$
where $|\cdot|$ denotes Lebesgue measure for sets, $|\ell|=\sum_{j}|\ell_j|$ for integer vectors, and $\langle\cdot,\cdot\rangle$ is the usual scalar product.}
\item[(\textbf{A2})]
There exist $d\geq1$ and $\delta<d-1$ such that $$\Omega_j(\xi)=j^d+\cdots+O(j^\delta),$$ where the dots stand for fixed lower order terms in $j$, allowing also negative
exponents. More precisely, there exists a fixed, parameter-independent sequence $\bar\Omega$ with $\bar\Omega_j=j^d+\cdots$ such that the tails $\tilde \Omega_j=\Omega_j-\bar\Omega_j$ give rise to a Lipschitz map $$\tilde\Omega:\Pi\rightarrow\ell_\infty^{-\delta},$$ where $\ell_\infty^{-\delta}$ is the space of all real sequences with finite norm $\pmb|\tilde\Omega\pmb|_{-\delta}=\sup_j|\tilde\Omega_j|j^{-\delta}$.
\item[(\textbf{A3})]
The perturbation $P$ is real analytic in the space
coordinates and Lipschitz in the parameters. For each $\xi\in \Pi$, its Hamiltonian vector space field $X_P=(P_y, -P_x, \sqrt{-1}P_{\bar u}, -\sqrt{-1}P_{u})^\top$ defines near $\mathcal{T}_0^n$ a real analytic map $$X_P:\mathcal{P}^{\textsf{a},\textsf{p}}\rightarrow\mathcal{P}^{\textsf{a},\bar {\textsf{p}}},~~~\left\{\begin{array}{cc}
\bar {\textsf{p}}\geq \textsf{p} ~~for ~d>1,\\
\bar {\textsf{p}}> \textsf{p} ~~for ~d=1.
\end{array}
\right.$$

\end{itemize}

Now, we can state our main result:
\begin{theorem}\label{th1}
{Consider the Hamiltonian (\ref{eq1}) and assume $(\textbf{A0})$-$(\textbf{A3})$ hold.
Let $m\geq L+\frac{\sqrt{4L^2+2L}}{2}$, $\tau\geq n-1$, $\mu$ is a fixed positive integer such that $(1+\frac{1}{2m})^\mu\geq2$, and
\begin{align}\label{a}
a=\left\{\begin{array}{lll}
[\frac{m+1}{3}]+1, ~if~ \frac{m+1}{3}~ is~ not~ integer,\\
\frac{m+1}{3},~~~~~~~~~if~ \frac{m+1}{3}~ is~ integer.
\end{array}\right.
\end{align}
 Then, there exists a sufficiently small $\varepsilon_0>0$,  a Cantor set $\Pi_\varepsilon\subset\Pi$, a Lipschitz continuous family of torus embeddings
$\Psi^*:\mathbb{T}^n\times\Pi_{\varepsilon}\rightarrow\mathcal{P}^{\textsf{a},\bar {\textsf{p}}}$, which is real analytic on $|\rm{Im}~ x|<\frac{s}{2}$ and close to the
identity, and a Lipschitz continuous map $\omega_*:\Pi_\varepsilon\rightarrow\mathbb{R}^n$,
such that
for each $\xi\in\Pi_\varepsilon$, the map $\Psi^*$ restricted to $\mathbb{T}^n\times\{\xi\}$ is a real analytic embedding of a rotational torus with frequencies $\omega_*(\xi)$ for the Hamiltonian $H$ at $\xi$. Moreover, the Cantor set $\Pi_{\e}$ and $\omega^*$ satisfy the following estimates,
\begin{eqnarray}
&&|{\rm meas}~(\Pi\setminus \Pi_{\e})|\to 0,\quad {\rm~as~}\e\to 0,\\
&&|\omega_*-\omega|_{\Pi}+\frac{\gamma}{2M}|\omega_*-\omega|_{\Pi}^{\mathcal{L}}\leq c\varepsilon^{\frac{3m}{32\mu(m+1)(m-a)(\tau+1)}}.
\end{eqnarray}
where $\gamma:=\varepsilon^{\frac{1}{32(2b)^{m+2}\mu(m+1)(m-a)(\tau+1)}}$ and  $M:=|\omega|_{\Pi}^{\mathcal{L}}+\pmb|\Omega\pmb|_{-\delta,\Pi}^{\mathcal{L}}<\infty$.

}
\end{theorem}

\begin{remark}\label{remark1}
{
We mention that, assumption (\textbf{A0}) allows $d$-dimensional degeneracy. Consider the infinite dimensional Hamiltonian (\ref{eq1}) with
the normal term $g(z)$ defined as follow
\begin{align}\label{gzlizi}
g(z)=g(w_0,\bar w_0):=\frac{1}{2p}\sum_{i=1}^{b}|w_{0i}|^{2p}+\frac{1}{2q}\sum_{i=1}^{b}|\bar w_{0i}|^{2q},
\end{align}
where $w_0=(w_{0,1},\cdots,w_{0,b}),\bar w_0=(\bar w_{0,1},\cdots,\bar w_{0,b})$, $p,q>1$.
The function in form of (\ref{gzlizi}) is the typical and common degenerate normal term which satisfies assumption (\textbf{A0}).
}

\end{remark}


\begin{remark}
As we mentioned above, the weak convexity condition in \textbf{(A0)} is indispensable. See below for a counter example:
\begin{proposition}\label{example}
Consider the infinite dimensional Hamiltonian (\ref{eq1}), for $b=1$,  with
$$\nabla g(z)=(\nabla g_1(w_0),\nabla g_2(\bar w_0)),~~~\varepsilon P=P_0(\varepsilon)\bar w_0,$$
where
\begin{align*}
\nabla g_1(w_0)&=\left\{\begin{array}{lll}
-(-w_0-1)^{\sigma},~~~&w_0\in(-2,-1),\\
0,~~~&w_0\in[-1,1],\\
(w_0-1)^{\sigma},~~~&w_0\in(1,2),
                           \end{array}\right.\\
\nabla g_2(\bar w_0)&=\bar w_0, ~~~\bar w_0\in(-2,2),
\end{align*}
$\sigma$ is a positive integer, and
$$P_0(\varepsilon)=\left\{\begin{array}{lll}
                              ~~~~~~~~~0,&&\varepsilon=0, \\
                              \varepsilon^\ell\sin\frac{1}{\varepsilon},&&\varepsilon\neq0,\,\ell\in \mathbb{Z}^+\setminus\{0\}.
                            \end{array}\right.$$
Then, the weak convexity condition in \textbf{(A0)} fails. Moreover, there is no low dimensional invariant tori can be preserved for any  small enough perturbation.
\end{proposition}
\end{remark}

The paper is organized as follows. In Section \ref{sec:result}, we state our main result, that is, Theorem \ref{th1}.
We will describe the quasi-linear iterative scheme, show the details of construction and estimates for one cycle of KAM steps.
Note that,  we prove the average equation can be solved by a equilibrium under assumption \textbf{(A0)} in subsection \ref{326}.
In Section \ref{sec:proof}, we complete the proof of Theorem \ref{th1} by deriving an iteration lemma and show the convergence of KAM iterations.
We show an example of Hamiltonian lattice equation as an application of  Theorem \ref{th1} in Section~\ref{sec:example}.
At last, we explain the indispensability of the weak convexity condition by counter example in the Appendix.

\section{KAM Step}\label{sec:KAM}
In this section, we will show the detailed construction and estimates for one cycle of KAM steps, which is essential to study the KAM theory, see \cite{chow,li,poschel1,poschel2}.

\subsection{Description of the 0-th KAM step.}


{{Recall the integer $m$ satisfying
\begin{align}\label{m}
m\geq L+\frac{\sqrt{4L^2+2L}}{2},
\end{align}
where $L\geq2$ is defined as in \textbf{(A0)}.
Denote $\Xi=8\mu(m+1)(m-a)(\tau+1)$, where $\mu$ is a fixed positive integer such that $(1+\frac{1}{2m})^\mu\geq2$. Then
\begin{equation}\label{gamma}\gamma=\varepsilon^{\frac{1}{4(2b)^{m+2}\Xi}}.\end{equation}}}
We consider $(\ref{eq1})$ and define the following $0$-th KAM step parameters:
\begin{align}\label{s0}
s_0&=s,~~\gamma_0=\gamma,
~~{\rho_0<\frac{s_0}{6}},~~\eta_0=\gamma_0^{2(2b)^{m+2}}\varepsilon^{\frac{1}{\Xi}},~~~r_0=\frac{s\gamma_0}{(K_1+1)^{\tau+1}},\\\label{h0y}
\omega_0(\xi)&=\omega(\xi),~~~\Omega_0(\xi)=\Omega(\xi),~~~g_0(z)=g(z),~~~f_0=0,~~~M_0=M,\\
D(s_0,r_0)&=\{(x,y,z,w,\bar{w}):|\textrm{Im}x|<s_0,|y|<r_0^2,||z||_{a,p}<r_0,||w||_{a,p}<r_0^a,||\bar{w}||_{a,p}<r_0^a\},
\end{align}
where $0<r_0,\gamma_0\leq 1$, $0\ll s\leq1$, $0<\eta_0\leq\frac{1}{16}$,
\begin{align}\label{K1}
K_1=([\log\frac{1}{\eta_0^{m+1}}]+1)^{3\mu}.
\end{align}
Therefore, we have that
\begin{align*}
H_0&=: H(x,y,u,\bar{u},\xi)=N_0+P_0,\\
N_0&=:\langle\omega_0(\xi), y\rangle+\langle w, \Omega_0(\xi)\bar{w}\rangle+g_0(z,\xi)+f_0,\\
P_0&=: \varepsilon P(x,y,u,\bar{u},\xi),
\end{align*}
with
$$\pmb\|X_{P_0}\pmb\|^{\lambda_0}_{r_0,D(s_0,r_0)}\leq \gamma_0^{2(2b)^{m+2}}r_0^{m-a}\eta_0^m,$$
where $0<\lambda_0<\frac{\gamma_0^{(2b)^{m+2}}}{M_0}$.

We first prove a crucial estimate.
\begin{lemma}
\begin{equation}\label{P0}
\pmb\|X_{P_0}\pmb\|^{\lambda_0}_{r_0,D(s_0,r_0)}\leq \gamma_0^{2(2b)^{m+2}}r_0^{m-a}\eta_0^m.
\end{equation}
\end{lemma}
\begin{proof}
{The proof is standard, but we give the explicit process due to the presence of the degenerate order $m$.}
Using the fact that $\gamma_0=\varepsilon^{\frac{1}{4(2b)^{m+2}\Xi}}$, $\eta_0=\gamma_0^{2(2b)^{m+2}}\varepsilon^{\frac{1}{\Xi}}$ and $[\log\frac{1}{\eta_0^{m+1}}]+1<\frac{1}{\eta_0^{m+1}}$, we have
\begin{align}\label{gre}
\gamma_0^{2(2b)^{m+2}}r_0^{m-a}\eta_0^m&>\frac{\gamma_0^{m-a+2(2b)^{m+2}}\eta_0^{3\mu(m+1)(m-a)(\tau+1)+m}}{2^{(\tau+1)(m-a)}}\\
&>\frac{\gamma_0^{m-a+2(2b)^{m+2}(1+m+3\mu(m+1)(m-a)(\tau+1))}\varepsilon^{\frac{3\mu(m+1)(m-a)(\tau+1)+m}{\Xi}}}{2^{(\tau+1)(m-a)}}\notag\\
&>\frac{\varepsilon^{\frac{1}{32\mu(2b)^{m+2}(m+1)(\tau+1)}+\frac{1}{16\mu(m-a)(\tau+1)}+\frac{3}{16}+\frac{3}{8}+\frac{1}{8\mu(m-a)(\tau+1)}}}{2^{(\tau+1)(m-a)}}\notag\\
&>\frac{\varepsilon^{\frac{1}{32\mu(2b)^{m+2}(m+1)(\tau+1)}+\frac{9}{16}}}{2^{(\tau+1)(m-a)}}.\notag
\end{align}
Moreover, let $\varepsilon_0>0$ be small enough so that
\begin{equation}\label{vare0}
\varepsilon_0^{\frac{1}{16}-\frac{1}{32\mu(2b)^{m+2}(m+1)(\tau+1)}}\pmb\|X_P\pmb\|^{\lambda_0}_{r_0,D(s_0,r_0)}\frac{2^{(\tau+1)(m-a)}}{s^{m-a}}\leq 1
\end{equation}
with $0<\lambda_0<\frac{\gamma_0^{(2b)^{m+2}}}{M_0}$
and for any $0<\varepsilon<\varepsilon_0$,
\begin{equation*}
\varepsilon^{\frac{1}{16}-\frac{1}{32\mu(2b)^{m+2}(m+1)(\tau+1)}}\pmb\|X_P\pmb\|^{\lambda_0}_{r_0,D(s_0,r_0)}\frac{2^{(\tau+1)(m-a)}}{s^{m-a}}\leq 1,
\end{equation*}
i.e.,
\begin{align}\label{P}
\varepsilon^{\frac{1}{16}}\pmb\|X_P\pmb\|^{\lambda_0}_{r_0,D(s_0,r_0)}\leq \frac{s^{m-a}\varepsilon^{\frac{1}{32\mu(2b)^{m+2}(m+1)(\tau+1)}}}{2^{(\tau+1)(m-a)}}.
\end{align}
Then by (\ref{gre}) and (\ref{P}),
\begin{align*}
\pmb\|X_{P_0}\pmb\|^{\lambda_0}_{r_0,D(s_0,r_0)}=\varepsilon^{\frac{9}{16}} \varepsilon^{\frac{7}{16}}\pmb\|X_P\pmb\|^{\lambda_0}_{r_0,D(s_0,r_0)}\leq \varepsilon^{\frac{9}{16}}\varepsilon^{\frac{6}{16}}\frac{s^{m-a}\varepsilon^{\frac{1}{32\mu(2b)^{m+2}(m+1)(\tau+1)}}}{2^{(\tau+1)(m-a)}}\leq\gamma_0^{2(2b)^{m+2}}r_0^{m-a}\eta_0^m,
\end{align*}
which implies (\ref{P0}).

The proof is complete.
\end{proof}

\subsection{Induction from  $\nu$-th KAM step}

\subsubsection{Description of the $\nu$-th KAM step}
We now define the $\nu$-th KAM step parameters:
\begin{align}\label{snu}
r_\nu&=\eta_{\nu-1}r_{\nu-1},~\eta_{\nu}=\eta_{\nu-1}^{1+\frac{1}{2m}},~{s_\nu=s_{\nu-1}-6\rho_{\nu-1}},~\rho_{\nu-1}=\frac{\rho_0}{2^{\nu-1}},\\\label{gammanu}
\gamma_\nu&=\frac{\gamma_{\nu-1}}{2}+\frac{\gamma_0}{4},~M_\nu=M_0(2-\frac{1}{2^\nu}).
\end{align}
Suppose that at $\nu$-th step, we have arrived at the following real analytic Hamiltonian:
\begin{align}\label{eq2}
H_\nu&=N_\nu+P_\nu,
\end{align}
with
{\begin{align}\label{Nnu}
N_\nu&={e_\nu(\xi)+\langle\omega_\nu(\xi),y\rangle+\langle w,\Omega_\nu(\xi)\bar w\rangle+g_\nu(z,\xi)+f_\nu(y,z,w,\bar w,\xi),}\\
g_\nu(z,\xi)&=g(z)+\sum_{j=0}^{\nu-1}\gamma_j^{2(2b)^{m+2}}r_j^{m-1}\eta_j^mO(\pmb\|z\pmb\|_{\textsf{a},\textsf{p}}^2),\\\label{fnu}
f_\nu(y,z,w,\bar w,\xi)&=\sum_{4\leq2|\imath'|\leq m}f_{\imath000} y^{\imath'}+\sum_{2|\imath'|+|\jmath'|\leq m,1\leq|\imath'|,|\jmath'|}f_{\imath'\jmath'00} y^{\imath'} z^{\jmath'}+\sum_{0<2|\imath'|+|\jmath'|\leq m}f_{\imath'\jmath'11} y^{\imath'} z^{\jmath'} w\bar w,
\end{align}}
defined on $D(s_\nu,r_\nu)$
and
\begin{equation}\label{XP}
\pmb\| X_{P_\nu}\pmb\|^{\lambda_\nu}_{r_\nu,D(s_\nu,r_\nu)}\leq\gamma_\nu^{2(2b)^{m+2}}r_\nu^{m-a}\eta_\nu^m, \end{equation}
with $0<\lambda_\nu<\frac{\gamma_\nu^{(2b)^{m+2}}}{M_\nu}$.

{Except for additional instructions, we will omit the index for all quantities of the present KAM step (at $\nu$-th step), use $+$ to index all quantities (Hamiltonian, domains, normal form, perturbation, transformation, etc.) in the next KAM step (at $(\nu+1)$-th step), and use $-$ to index all quantities in the previous KAM step (at $(\nu-1)$-th step).}
To simplify the notations, we will not specify the dependence of $P$, $P_+$ etc. All the constants {{$c_1$-$c_3$}} below are positive and independent of the iteration process, and we will also use $c$ to denote any intermediate positive constant which is independent of the iteration process.

Define
\begin{align*}
\eta_+&=\eta^{1+\frac{1}{2m}},\\
r_{+}&=\eta r,\\
\rho_{+}&=\frac{\rho}{2},\\
s_{+}&=s-6\rho,\\
\sigma_+&=\frac{\sigma}{2}+\frac{\sigma_0}{4},\\
\gamma_+&=\frac{\gamma}{2}+\frac{\gamma_0}{4},\\
M_\nu&=M_0(2-\frac{1}{2^\nu}),\\
K_{+}&=([\log(\frac{1}{\eta^{m+1}})]+1)^{3\mu},\\
D_{+}&=D(s_{+}, r_{+}),\\
{\Pi_+}&=\{\xi\in\Pi:|\langle k,\omega(\xi)\rangle+\langle\ell,\Omega(\xi)\rangle|\geq\frac{\gamma\langle \ell\rangle_d}{(1+|k|)^\tau}, |k|\leq K_+, |\ell|\leq2, |k|+|\ell|\neq0 \}.\\
\end{align*}

\subsubsection{Construct a symplectic transformation}

We will construct a symplectic transformation $\Phi_{+}:D(s_{+},r_{+})\rightarrow D(s,r)$
such that it transforms the Hamiltonian ($\ref{eq2}$) into the Hamiltonian of the next KAM cycle (at ($\nu$+1)-th step), i.e.,
\begin{equation}\label{H+}
H_{+}=H\circ\Phi_{+}=N_{+}+P_{+},
\end{equation}
where $N_{+}$ and $P_{+}$ have similar properties as $N$ and $P$ respectively on $D(s_{+},r_{+})$.


Next, we show the detailed construction of $\Phi_+$ and the estimate of $P_+$.

\subsubsection{Truncation}
Consider the Taylor-Fourier series of $P$:
\begin{equation*}
P=\sum_{k\in \mathbb{Z}^n,~\imath,\jmath,\ell_1,\ell_2\in \mathbb{Z}_+^n,\ell=(\ell_1,\ell_2)}p_{k\imath\jmath\ell}y^{\imath}z^{\jmath} w^{\ell_1}\bar{w}^{\ell_2} {\rm e}^{\sqrt{-1}\langle k,x\rangle},
\end{equation*}
and let $R$ be the truncation of $P$ of the form
\begin{align*}
R&=\sum_{|k|\leq K_+,~2|\imath|+|\jmath|\leq m,~|\ell|=|\ell_1|+|\ell_2|\leq2}p_{k\imath\jmath\ell}y^{\imath}z^{\jmath} w^{\ell_1}\bar{w}^{\ell_2} {\rm e}^{\sqrt{-1}\langle k,x\rangle},\\
[R]&=\sum_{2|\imath|+|\jmath|\leq m,~|\ell|=|\ell_1|+|\ell_2|\leq2}p_{0\imath\jmath\ell}y^{\imath}z^{\jmath} w^{\ell_1}\bar{w}^{\ell_2}.
\end{align*}
Next we will prove that the norm of $X_P-X_R$ is much smaller than the norm of $X_P$ by selecting truncation appropriately, see the below lemma.

\begin{lemma}\label{le1}
Assume that
$$\textbf{\textsc{(H1)}}~ K_+^n{\rm e}^{-K_+\rho} <\eta^{m+1}.$$
Then there is a constant $c_1$ such that
\begin{align}\label{XP-R}
\pmb\|X_P-X_R\pmb\|^*_{\eta r,D(s-\rho,8\eta r)}&\leq c_1\eta^{m-a+1} \pmb\|X_P\pmb\|^*_{r,D(s,r)},\\\label{XR}
\pmb\|X_R\pmb\|^*_{r,D(s-\rho,8\eta r)}&\leq c_1 \pmb\|X_P\pmb\|^*_{r,D(s,r)}.
\end{align}
\end{lemma}
\begin{proof}
{Notice the $m$ order degeneracy and the choice of $a$. }Denote
\begin{align*}
I&=\sum_{2|\imath|+|\jmath|\geq m+1}p_{k\imath\jmath\ell}y^{\imath}z^{\jmath} w^{\ell_1}\bar{w}^{\ell_2} {\rm e}^{\sqrt{-1}\langle k,x\rangle},~~~~II=\sum_{|k|>K_+,2|\imath|+|\jmath|\leq m,|\ell_1|+|\ell_2|\leq2}p_{k\imath\jmath\ell}y^{\imath}z^{\jmath} w^{\ell_1}\bar{w}^{\ell_2} {\rm e}^{\sqrt{-1}\langle k,x\rangle},\\
III&=\sum_{|k|\leq K_+,2|\imath|+|\jmath|\leq m,|\ell_1|+|\ell_2|\geq3}p_{k\imath\jmath\ell}y^{\imath}z^{\jmath}w^{\ell_1}\bar{w}^{\ell_2} {\rm e}^{\sqrt{-1}\langle k,x\rangle}.
\end{align*}
Then
\begin{align}\label{P-R}
P-R=I+II+III.
\end{align}

We claim that
\begin{align}\label{XI}
\pmb\|X_{I}\pmb\|_{\eta r,D(s,8\eta r)}&=\sup_{D(s,8\eta r)}\{\frac{|I_y|}{(\eta r)^{a-2}}+\frac{|I_x|}{(\eta r)^a}+\frac{||I_{z}||_{\bar{p}}}{(\eta r)^{a-1}}+||I_{w}||_{\bar{p}}+||I_{\bar{w}}||_{\bar{p}}\}\\
&<c\eta^{m+1-a} \pmb\|X_P\pmb\|_{r,D(s,r)},\notag
\end{align}
and
\begin{align}\label{XI*}
\pmb\|X_{I}\pmb\|^{\mathcal{L}}_{\eta r,D(s,8\eta r)}<c\eta^{m+1-a} \pmb\|X_P\pmb\|^{\mathcal{L}}_{r,D(s,r)}.
\end{align}
Indeed,
\begin{align*}
|I_x|_{D(s,8\eta r)}&=|\sum_{2|\imath|+|\jmath|\geq m+1}\frac{\partial p_{k\imath\jmath\ell}{\rm e}^{\sqrt{-1}\langle k,x\rangle}}{\partial x}y^\imath z^{\jmath}w^{\ell_1}\bar w^{\ell_2}|_{D(s,8\eta r)}\\
&\leq c\sum_{2|\imath|+|\jmath|\geq m+1}\frac{|8\eta r|^{2|\imath|+|\jmath|+|\ell|}|P_x|_{D(s,r)}}{(r-8\eta r)^{2|\imath|+|\jmath|+|\ell|}}\\
&\leq c\sum_{2|\imath|+|\jmath|\geq m+1}\eta^{m+1}r^a\pmb\|X_P\pmb\|_{r,D(s,r)},
\end{align*}
where the second inequality follows from Cauchy estimate and the last inequality follows from the definition of $\pmb\|X_P\pmb\|_{r,D(s,r)}$. Namely,
\begin{align*}
\frac{|I_x|_{D(s,8\eta r)}}{(\eta r)^a}\leq c\eta^{m+1-a}\pmb\|X_P\pmb\|_{r,D(s,r)}.
\end{align*}
For any $\xi,\zeta\in\Pi$,
\begin{align*}
\frac{|I_x|^{\mathcal{L}}_{D(s,8\eta r)}}{(\eta r)^a}&=
\frac{|I_x(\xi)-I_x(\zeta)|_{D(s,8\eta r)}}{|\xi-\zeta|(\eta r)^a}\\
&=\frac{1}{|\xi-\zeta|(\eta r)^a}|\sum_{2|\imath|+|\jmath|\geq m+1}\frac{\partial (p_{k\imath\jmath\ell}(\xi)-p_{k\imath\jmath\ell}(\zeta)){\rm e}^{\sqrt{-1}\langle k,x\rangle}}{\partial x}y^\imath w_0^{\jmath_1}\bar{w}_0^{\jmath_2}w^{\ell_1}\bar w^{\ell_2}|_{D(s,8\eta r)}\\
&\leq c\sum_{2|\imath|+|\jmath|\geq m+1}\frac{|8\eta r|^{2|\imath|+|\jmath|+|\ell|}|P_x(\xi)-P_x(\zeta)|_{D(s,r)}}{|\xi-\zeta|(\eta r)^a(r-8\eta r)^{2|\imath|+|\jmath|+|\ell|}}\\
&\leq c\eta^{m+1-a}\pmb\|X_P\pmb\|_{r,D(s,r)}^{\mathcal{L}}.
\end{align*}

Similarly, we can prove
\begin{align*}
\frac{|I_y|_{D(s,8\eta r)}}{(\eta r)^{a-2}},~~\frac{||I_{z}||_{\bar{p}}}{(\eta r)^{a-1}},~~ ||I_{w}||_{\bar{p}},~~||I_{\bar{w}}||_{\bar{p}}\leq c\eta^{m+1-a}\pmb\|X_P\pmb\|_{r,D(s,r)},
\end{align*}
and
\begin{align*}
\frac{|I_y|^{\mathcal{L}}_{D(s,8\eta r)}}{(\eta r)^{a-2}},~~\frac{||I_{z}||^{\mathcal{L}}_{\bar{p}}}{(\eta r)^{a-1}},~~ ||I_{w}||^{\mathcal{L}}_{\bar{p}},~~||I_{\bar{w}}||^{\mathcal{L}}_{\bar{p}}\leq c\eta^{m+1-a}\pmb\|X_P\pmb\|^{\mathcal{L}}_{r,D(s,r)}.
\end{align*}
Thus (\ref{XI}) and (\ref{XI*}) hold.

We claim that
\begin{align}\label{XIII}
\pmb\|X_{III}\pmb\|^*_{\eta r,D(s,8\eta r)}
&<c\eta^{m+1-a} \pmb\|X_P\pmb\|^*_{r,D(s,r)}.
\end{align}
Indeed,
\begin{align*}
|III_x|_{D(s,8\eta r)}&=|\sum_{|k|\leq K_+,2|\imath|+|\jmath|\leq m,|\ell_1|+|\ell_2|\geq3}\frac{\partial p_{k\imath\jmath\ell}{\rm e}^{\sqrt{-1}\langle k,x\rangle}}{\partial x}y^\imath z^{\jmath}w^{\ell_1}\bar w^{\ell_2}|_{D(s,8\eta r)}\\
&\leq c\sum_{|k|\leq K_+,2|\imath|+|\jmath|\leq m,|\ell_1|+|\ell_2|\geq3}\frac{|8\eta r|^{2|\imath|+|\jmath|+a|\ell|}|P_x|_{D(s,r)}}{(r-8\eta r)^{2|\imath|+|\jmath|+a|\ell|}}\\
&\leq c \eta^{3a}r^{a}\pmb\|X_P\pmb\|_{r,D(s,r)},\\
\end{align*}
i.e.,
{{\begin{align*}
\frac{|III_x|_{D(s,8\eta r)}}{(\eta r)^a}\leq c\eta^{2a}\pmb\|X_P\pmb\|_{r,D(s,r)}\leq  c\eta^{m+1-a}\pmb\|X_P\pmb\|_{r,D(s,r)}.
\end{align*}}}
Similarly, we can prove
\begin{align*}
\frac{|III_y|_{D(s,8\eta r)}}{(\eta r)^{a-2}},~~\frac{||III_{z}||_{\bar{p}}}{(\eta r)^{a-1}},~~ ||III_{w}||_{\bar{p}},~~||III_{\bar{w}}||_{\bar{p}}\leq c\eta^{m+1-a}\pmb\|X_P\pmb\|_{r,D(s,r)},
\end{align*}
and
\begin{align*}
\frac{|III_y|^{\mathcal{L}}_{D(s,8\eta r)}}{(\eta r)^{a-2}},~~\frac{||III_{z}||^{\mathcal{L}}_{\bar{p}}}{(\eta r)^{a-1}},~~ {||III_{w}||^{\mathcal{L}}_{\bar{p}}},~~{||III_{\bar{w}}||^{\mathcal{L}}_{\bar{p}}}\leq c\eta^{m+1-a}\pmb\|X_P\pmb\|_{r,D(s,r)}.
\end{align*}

We now estimate $\pmb\|X_{II}\pmb\|^*_{\eta r,D(s-\rho,8\eta r)}$. According to the definition of $II$, the Lemma A.$2$. in \cite{poschel2} and (\ref{XI}), we have
\begin{align*}
|II_x|_{D(s-\rho,8\eta r)}&=|\frac{\partial(P-I-III)}{\partial x}-\frac{\partial R}{\partial x}|_{D(s-\rho,8\eta r)}\leq cK_+^n{\rm e}^{-K_+\rho}|\frac{\partial(P-I-III)}{\partial x}|_{D(s,8\eta r)}\\
&\leq cK_+^n{\rm e}^{-K_+\rho}(r^a\pmb\|X_P\pmb\|_{r,D(s,r)}+\eta^{m+1}r^a\pmb\|X_{P}\pmb\|_{r,D(s,r)}+\eta^{m+1}r^a\pmb\|X_{P}\pmb\|_{ r,D(s, r)})\\
&\leq cK_+^n{\rm e}^{-K_+\rho}r^a(1+\eta^{m+1}+\eta^{m+1})\pmb\|X_{P}\pmb\|_{r,D(s,r)},
\end{align*}
i.e.,
\begin{align}\label{IIx}
\frac{|II_x|_{D(s-\rho,8\eta r)}}{(\eta r)^a}&\leq c\frac{K_+^n{\rm e}^{-K_+\rho}}{\eta^a}\pmb\|X_{P}\pmb\|_{r,D(s,r)}\leq c\eta^{m+1-a}\pmb\|X_{P}\pmb\|_{r,D(s,r)},
\end{align}
where the last inequality follows from (\textbf{H1}).
Similarly, we can get
\begin{align}\label{IIy}
\frac{|II_y|_{D(s-\rho,8\eta r)}}{(\eta r)^{a-2}},~~\frac{||II_{z}||_{\bar{p}}}{(\eta r)^{a-1}},~~ ||II_{w}||_{\bar{p}},~~||II_{\bar{w}}||_{\bar{p}}\leq c\eta^{m+1-a}\pmb\|X_P\pmb\|_{r,D(s,r)},
\end{align}
and
\begin{align}\label{IIy*}
\frac{|II_y|^{\mathcal{L}}_{D(s-\rho,8\eta r)}}{(\eta r)^{a-2}},~~\frac{||II_{z}||^{\mathcal{L}}_{\bar{p}}}{(\eta r)^{a-1}},~~ ||II_{w}||^{\mathcal{L}}_{\bar{p}},~~||II_{\bar{w}}||^{\mathcal{L}}_{\bar{p}}\leq c\eta^{m+1-a}\pmb\|X_P\pmb\|_{r,D(s,r)}.
\end{align}
Then by (\ref{IIx}), (\ref{IIy}) and (\ref{IIy*})
\begin{align}\label{XII}
\pmb\|X_{II}\pmb\|^*_{\eta r,D(s-\rho,8\eta r)}
<c\eta^{m+1-a} \pmb\|X_P\pmb\|^*_{r,D(s,r)}.
\end{align}
Therefore, it follows from (\ref{P-R}), (\ref{XI}),  (\ref{XIII}) and (\ref{XII}) that
\begin{align*}
\pmb\|X_P-X_R\pmb\|^*_{\eta r,D(s-\rho,8\eta r)}&\leq c\eta^{m+1-a} \pmb\|X_P\pmb\|^*_{r,D(s,r)},
\end{align*}
and
\begin{align*}
\pmb\|X_R\pmb\|^*_{r,D(s-\rho,8\eta r)}
&\leq \pmb\|X_P\pmb\|^*_{r,D(s-\rho,8\eta r)}+ \pmb\|X_{I}+X_{II}+X_{III}\pmb\|^*_{r,D(s-\rho,8\eta r)}\\
&\leq \pmb\|X_P\pmb\|^*_{r,D(s-\rho,8\eta r)}+ \pmb\|X_{P}-X_{R}\pmb\|^*_{r,D(s-\rho,8\eta r)}\\
&\leq\pmb\|X_P\pmb\|^*_{r,D(s,r)}+c\eta^{m-a+1}\pmb\|X_P\pmb\|^*_{r,D(s,r)}\\
&\leq c\pmb\|X_P\pmb\|^*_{r,D(s,r)}.
\end{align*}
\end{proof}

\subsubsection{Homological Equation}

We first construct a generalized Hamiltonian $F$ of the form
\begin{align}\label{eq3}
F&=\sum_{|k|\leq K_+,~2|\imath|+|\jmath|\leq m,|\ell_1|+|\ell_2|\leq2, |k|+|\ell_1|+|\ell_2|\neq0}F_{k\imath\jmath\ell_1\ell_2}y^{\imath}z^\jmath w^{\ell_1}\bar w^{\ell_2} {\rm e}^{\sqrt{-1}\langle k,x\rangle},
\end{align}
which satisfies the equation
\begin{equation}\label{eq4}
\{N,F\}+R-[R]-Q=0,
\end{equation}
where $[R]=\frac{1}{(2\pi)^n}\int_{\mathbb{T}^n}R(x,y,z,w,\bar w){\rm d}x$ is the average of the truncation $R$, and the correction term
\begin{align}\label{Q}
Q=(\partial_zg+\partial_zf)J\partial_zF|_{2|\imath|+|\jmath|> m, or |\ell|> 2}.
\end{align}
{
Notice that
\begin{align}\label{NF}
\{N,F\}=-\partial_yN\partial_xF+\partial_xN\partial_yF+\partial_zNJ\partial_zF-\sqrt{-1}\partial_{\bar w}N\partial_{w}F+\sqrt{-1}\partial_wN\partial_{\bar w}F.
\end{align}
Recall that
\begin{align*}
f&=\sum_{4\leq2|\imath'|\leq m}f_{\imath000} y^{\imath'}+\sum_{2|\imath'|+|\jmath'|\leq m,1\leq|\imath'|,|\jmath'|}f_{\imath'\jmath'00} y^{\imath'} z^{\jmath'}+\sum_{0<2|\imath'|+|\jmath'|\leq m}f_{\imath'\jmath'11} y^{\imath'} z^{\jmath'} w\bar w\\
&=:f_1+f_2,
\end{align*}
where $f_1=\sum_{4\leq2|\imath'|\leq m}f_{\imath000} y^{\imath'}$ and $f_2=f-f_1$,
and
$$g_\nu=\sum_{2\leq|\beta|\leq m}g_{0\beta00}z^{\beta},~~~~~1\leq\nu.$$
Substituting (\ref{Nnu}) and (\ref{eq3}) into (\ref{NF}), we can get
\begin{align}\label{NF1}
\{N,F\}=&-\sqrt{-1}\langle\omega+\partial_yf_1,k\rangle F_{k\imath\jmath\ell_1\ell_2}y^{\imath}z^\jmath w^{\ell_1}\bar w^{\ell_2} {\rm e}^{\sqrt{-1}\langle k,x\rangle}+\partial_yf_2\partial_xF+(\partial_zg+\partial_zf)J\partial_zF\notag\\
&-\sqrt{-1}\langle\Omega,\ell_1-\ell_2\rangle F_{k\imath\jmath\ell_1\ell_2}y^{\imath}z^\jmath w^{\ell_1}\bar w^{\ell_2} {\rm e}^{\sqrt{-1}\langle k,x\rangle}\\
&-\sqrt{-1}\langle\partial_{w\bar w}f,\ell_1-\ell_2\rangle F_{k\imath\jmath\ell_1\ell_2}y^{\imath}z^\jmath w^{\ell_1}\bar w^{\ell_2} {\rm e}^{\sqrt{-1}\langle k,x\rangle}.\notag
\end{align}

In order to simplify the notations, we sometimes omit the subscript of $\sum$ and only use $\sum$ to represent the sum to the index over the corresponding range of variation.

We begin to calculate the second term (\ref{NF1}):
\begin{align}\label{A}
\partial_y f_2 \partial_x F=&\sum_{2|\imath'|+|\jmath'|\leq m;1\leq|\imath'|,|\jmath'|;|\ell'|=0,1}f_{\imath'\jmath'\ell'\ell'} \partial_y(y^{\imath'}) z^{\jmath'} w^{\ell'}\bar w^{\ell'} \\
&\sum_{|k|\leq K_+,~2|\imath|+|\jmath|\leq m,|\ell_1|+|\ell_2|\leq2, |k|+|\ell_1|+|\ell_2|\neq0}F_{k\imath\jmath\ell_1\ell_2}y^{\imath}z^\jmath w^{\ell_1}\bar w^{\ell_2} \partial_x({\rm e}^{\sqrt{-1}\langle k,x\rangle})\notag\\
=&\sum\mathcal{A}_{\textsf{b}(\textsf{c}+1)\ell'\ell'}F_{k(\imath-\textsf{b}+1)(\jmath-\textsf{c}-1)(\ell_1-\ell')(\ell_2-\ell')}y^{\imath}z^\jmath w^{\ell_1}\bar w^{\ell_2} {\rm e}^{\sqrt{-1}\langle k,x\rangle},\notag
\end{align}
where $\mathcal{A}_{\textsf{b}(\textsf{c}+1)\ell'\ell'}$ is concerned with $\partial_y f_2$ and $k$, and $2|\textsf{b}|+\textsf{c}+1\leq m$, $1\leq|\textsf{b}|$, $0\leq |\textsf{c}|$, $|\ell'|=0,1$, ${|k|\leq K_+,~2|\imath|+|\jmath|\leq m,|\ell_1|+|\ell_2|\leq2, |k|+|\ell_1|+|\ell_2|\neq0}$.

Next we calculate the third term of (\ref{NF1}):
\begin{align}\label{mathcalB}
(\partial_zg+\partial_zf)J\partial_zF&=(\partial_zg+\partial_zf)J\partial_zF|_{2|\imath|+|\jmath|\leq m, |\ell|\leq 2}+(\partial_zg+\partial_zf)J\partial_zF|_{2|\imath|+|\jmath|> m, or |\ell|> 2}\\
&=:(\partial_zg+\partial_zf)J\partial_zF|_{2|\imath|+|\jmath|\leq m, |\ell|\leq 2}+Q.\notag
\end{align}
Specially,
\begin{align*}
\partial_zgJ\partial_zF|_{2|\imath|+|\jmath|\leq m, |\ell|\leq 2}
&=\partial_zgJ\sum_{|k|\leq K_+,~2|\imath|+|\jmath|\leq m,|\ell_1|+|\ell_2|\leq2, |k|+|\ell_1|+|\ell_2|\neq0} F_{k\imath\jmath\ell_1\ell_2}y^{\imath}\partial_z(z^{\jmath}) w^{\ell_1}\bar w^{\ell_2} {\rm e}^{\sqrt{-1}\langle k,x\rangle}\\
&=:\sum_{}S_{\imath\jmath\ell}F_{k\imath\jmath\ell_1\ell_2}y^\imath z^{\jmath}w^{\ell_1}\bar w^{\ell_2}{\rm e}^{\sqrt{-1}\langle k,x\rangle}+\sum_{}\mathcal{B}_iF_{k\imath(\jmath-i+1)\ell_1\ell_2}y^\imath z^{\jmath}w^{\ell_1}\bar w^{\ell_2}{\rm e}^{\sqrt{-1}\langle k,x\rangle} ,
\end{align*}
where $S_{\imath\jmath\ell}$ is a ($|\imath|+|\jmath|+|\ell|$) order tensor and concerned with $\partial_z^2g(0)$ and $J$, $\mathcal{B}_i$  is concerned with $\partial_z^2g(z)-\partial_z^2g(0)$ and $J$, and $2\leq i$, ${|k|\leq K_+,~2|\imath|+|\jmath|\leq m,|\ell_1|+|\ell_2|\leq2, |k|+|\ell_1|+|\ell_2|\neq0}$.
And
\begin{align*}
\partial_zfJ\partial_zF|_{2|\imath|+|\jmath|\leq m, |\ell|\leq 2}=&\partial_zf
J
\sum_{|k|\leq K_+,~2|\imath|+|\jmath|\leq m,|\ell_1|+|\ell_2|\leq2, |k|+|\ell_1|+|\ell_2|\neq0} F_{k\imath\jmath\ell_1\ell_2}y^{\imath}\partial _z(z^{\jmath}) w^{\ell_1}\bar w^{\ell_2} {\rm e}^{\sqrt{-1}\langle k,x\rangle}\\
=:&\sum \mathcal{B}_{\textsf{b}(\textsf{c}+1)\ell'\ell'}F_{k(\imath-\textsf{b})(\jmath-\textsf{c}+1)(\ell_1-\ell')(\ell_2-\ell')}y^{\imath}z^{\jmath}w^{\ell_1}\bar w^{\ell_2}{\rm e}^{\sqrt{-1}\langle k,x\rangle},
\end{align*}
where $\mathcal{B}_{\textsf{b}(\textsf{c}+1)\ell'\ell'}$ is concerned with $\partial_{y}^{\textsf{b}}\partial_z^{\textsf{c}+2}\partial_w^{\ell'}\partial_{\bar w}^{\ell'}f(0,0,0,0)$, $|\ell'|=0, 1$ and $J$, and $2|\textsf{b}|+\textsf{c}+1\leq m$, $1\leq|\textsf{b}|$, $0\leq |\textsf{c}|$, ${|k|\leq K_+,~2|\imath|+|\jmath|\leq m,|\ell_1|+|\ell_2|\leq2, |k|+|\ell_1|+|\ell_2|\neq0}$.

Now, we calculate the fifth term of (\ref{NF1}):
\begin{align}\label{C}
&-\sum_{|k|\leq K_+,~2|\imath|+|\jmath|\leq m,|\ell_1|+|\ell_2|\leq2, |k|+|\ell_1|+|\ell_2|\neq0}\sqrt{-1}\langle\partial_{w\bar w}f,\ell_1-\ell_2\rangle F_{k\imath\jmath\ell_1\ell_2}y^{\imath}z^\jmath w^{\ell_1}\bar w^{\ell_2} {\rm e}^{\sqrt{-1}\langle k,x\rangle}\\
&=\sum \mathcal{C}_{\textsf{b}(\textsf{c}+1)}F_{k(\imath-\textsf{b})(\jmath-\textsf{c}-1)\ell_1\ell_2}y^{\imath}z^\jmath w^{\ell_1}\bar w^{\ell_2} {\rm e}^{\sqrt{-1}\langle k,x\rangle},\notag
\end{align}
where $\mathcal{C}_{\textsf{b}(\textsf{c}+1)}$ is concerned with $\partial_y^{\textsf{b}}\partial_z^{\textsf{c}+1}\partial_{w}\partial_{\bar w}f(0,0,0,0)$, $2|\textsf{b}|+\textsf{c}+1\leq m$, $1\leq|\textsf{b}|$, $0\leq |\textsf{c}|$, ${|k|\leq K_+,~2|\imath|+|\jmath|\leq m,|\ell_1|+|\ell_2|\leq2, |k|+|\ell_1|+|\ell_2|\neq0}$.

Substituting (\ref{A}), (\ref{mathcalB}) and (\ref{C}) into (\ref{NF1}), combining (\ref{NF1}) with (\ref{eq4}),  and
comparing the coefficients above, we then obtain the following quasi-linear equations:

\begin{itemize}
\item[1.] For $\ell_1=\ell_2$, $|\ell_1|=0$, $k\neq0$,
    \begin{align}\label{he1}
    &\sqrt{-1}(\langle\omega+\partial_yf_1,k\rangle I_{(2b)^{\imath+\jmath}}+S_{\imath\jmath\ell})F_{k\imath\jmath00}+\mathcal{A}_{\textsf{b}(\textsf{c}+1)00}F_{k(\imath-\textsf{b}+1)(\jmath-\textsf{c}-1)00}\\
    &+\mathcal{B}_{\textsf{b}(\textsf{c}+1)00}F_{k(\imath-\textsf{b})(\jmath-\textsf{c}+1)00}+\mathcal{B}_iF_{k\imath(\jmath-i+1)00}+\mathcal{C}_{\textsf{b}(\textsf{c}+1)}F_{k(\imath-\textsf{b})(\jmath-\textsf{c}-1)00}=p_{k\imath\jmath00}.\notag\end{align}
\item[2.] For $\ell_1\neq\ell_2$, \begin{align}\label{he2}
    &\sqrt{-1}((\langle\omega+\partial_yf_1,k\rangle +\langle\Omega,\ell_1-\ell_2\rangle) I_{(2b)^{\imath+\jmath+\ell_1+\ell_2}}+S_{\imath\jmath\ell})F_{k\imath\jmath\ell_1\ell_2}\\
    &+\mathcal{A}_{\textsf{b}(\textsf{c}+1)\ell'\ell'}F_{k(\imath-\textsf{b}+1)(\jmath-\textsf{c}-1)(\ell_1-\ell')(\ell_2-\ell')}+\mathcal{B}_{\textsf{b}(\textsf{c}+1)\ell'\ell'}F_{k(\imath-\textsf{b})(\jmath-\textsf{c}+1)(\ell_1-\ell')(\ell_2-\ell')}+\mathcal{B}_iF_{k\imath(\jmath-i+1)\ell_1\ell_2}\notag\\
    &+\mathcal{C}_{\textsf{b}(\textsf{c}+1)}F_{k(\imath-\textsf{b})(\jmath-\textsf{c}-1)\ell_1\ell_2}=p_{k\imath\jmath\ell_1\ell_2}\notag.
    \end{align}
\item[3.] For $\ell_1=\ell_2$, $|\ell_1|=1$, $k\neq0$, \begin{align}\label{he3}
    &\sqrt{-1}(\langle\omega+\partial_yf_1,k\rangle I_{(2b)^{\imath+\jmath+2}}+S_{\imath\jmath\ell})F_{k\imath\jmath11}+\mathcal{A}_{\textsf{b}(\textsf{c}+1)00}F_{k(\imath-\textsf{b}+1)(\jmath-\textsf{c}-1)11}\\
    &+\mathcal{A}_{\textsf{b}(\textsf{c}+1)11}F_{k(\imath-\textsf{b}+1)(\jmath-\textsf{c}-1)00}+\mathcal{B}_{\textsf{b}(\textsf{c}+1)00}F_{k(\imath-\textsf{b})(\jmath-\textsf{c}+1)11}+\mathcal{B}_{\textsf{b}(\textsf{c}+1)11}F_{k(\imath-\textsf{b})(\jmath-\textsf{c}+1)00}\notag\\
    &+\mathcal{B}_iF_{k\imath(\jmath-i+1)11}+\mathcal{C}_{\textsf{b}(\textsf{c}+1)}F_{k(\imath-\textsf{b})(\jmath-\textsf{c}-1)11}=p_{k\imath\jmath11}.\notag\end{align}
\end{itemize}
Here  the notations $\textsf{b}, \textsf{c}, i, \ell', \imath, \jmath, k$ are defined as above, $\Omega=(\Omega^j), j\in\mathbb{N}_+\setminus\mathcal{N}$, \\ $\mathcal{A}_{\textsf{b}(\textsf{c}+1)00}F_{k(\imath-\textsf{b}+1)(\jmath-\textsf{c}-1)00}$ stands for $\sum_{2|\textsf{b}|+\textsf{c}+1\leq m, 1\leq|\textsf{b}|, 0\leq |\textsf{c}|}\mathcal{A}_{\textsf{b}(\textsf{c}+1)00}F_{k(\imath-\textsf{b}+1)(\jmath-\textsf{c}-1)00}$, and the rest of terms are analogously defined.
}

We declare that the quasi-linear equations (\ref{he1})-(\ref{he3}) are solvable under some suitable conditions. We denote
$${\Pi_+}=\{\xi\in\Pi:|\langle k,\omega(\xi)\rangle+\langle\ell,\Omega(\xi)\rangle|\geq\frac{\gamma\langle \ell\rangle_d}{(1+|k|)^\tau}, |k|\leq K_+, |\ell|\leq2, |k|+|\ell|\neq0 \}.$$
Then we can solve equations (\ref{he1})-(\ref{he3}) on $\Pi_+$. The details can be seen in the following lemma:
\begin{lemma}\label{le2}
Assume that
\begin{align*}
&{\textbf{\textsc{(H2)}} ~8r<\frac{\langle\ell\rangle_d(\gamma-\gamma_+)}{(K_++1)^{\tau+1}}.} 
\end{align*}
Then the quasi-linear equations (\ref{he1})-(\ref{he3}) have a solution $F_{k\imath\jmath\ell}$ satisfying
\begin{align}\label{XF}
\pmb\|X_{F}\pmb\|^{\lambda}_{r,D(s-\rho,{r})}
&\leq c_2A_\rho r^{m-a}\eta^m,
\end{align}
where $0<\lambda<\frac{\gamma^{(2b)^{\imath+\jmath+\ell}}}{M}$, and
\begin{align}\label{Arho}
A_\rho^2=(\sum_{0<|k|\leq K_+,~2|\imath|+|\jmath|\leq m, ~|\ell|\leq2}(\frac{(1+|k|)^{1+(2b)^{\imath+\jmath+\ell}\tau}}{(\langle\ell\rangle_d)^{(2b)^{\imath+\jmath+\ell}}})^2{\rm e}^{-2|k|\rho}).
\end{align}
Moreover,
\begin{align*}
\pmb\|DX_{F}\pmb\|_{r,r,D(s-2\rho,\frac{r}{2})}\leq c_2\frac{1}{\rho r^a}\pmb\|X_F\pmb\|_{r,D(s-\rho,r)}.
\end{align*}
\end{lemma}

\begin{proof}
{For $\forall y\in D(r)$, by \textbf{(H2)}, \begin{align*}
|\partial_yf_1|&\leq cr<\frac{\gamma\langle\ell\rangle_d}{8(|k|+1)^{\tau+1}}.
\end{align*}
Denote $$L_k=\langle k,\omega+\partial_yf_1\rangle+\langle \ell,\Omega\rangle+\tilde\lambda,$$
where $\tilde\lambda$ is the minimum in the absolute value of the eigenvalue of $ S_{\imath\jmath\ell}$, and $$|\tilde\lambda|\leq|\partial_z^2g(0)|\leq\gamma_-^{2(2b)^{m+2}}r_-^{m-2}\eta_-^m\leq\gamma_-^{2(2b)^{m+2}}r^{m-2}\eta^2\leq\frac{\gamma\langle\ell\rangle_d}{8(|k|+1)^{\tau+1}}.$$
Then
\begin{align*}
|L_k|&=|\langle k,\omega\rangle+\langle\ell,\Omega\rangle|-|\tilde\lambda|-|\langle k,\partial_yf_1\rangle|\\
&\geq\frac{\gamma\langle\ell\rangle_d}{(1+|k|)^\tau}-\frac{\gamma\langle\ell\rangle_d}{8(1+|k|)^\tau}-\frac{\gamma\langle\ell\rangle_d}{8(1+|k|)^\tau}\\
&\geq\frac{\gamma\langle\ell\rangle_d}{2(1+|k|)^\tau},
\end{align*}
and
\begin{align}\label{LI}
|\det L_kI_{(2b)^{\imath+\jmath+\ell}}|\geq(\frac{\gamma\langle\ell\rangle_d}{2(1+|k|)^\tau})^{(2b)^{\imath+\jmath+\ell}}.
\end{align}
Define the coefficient matrix of (\ref{he1})-(\ref{he3}) by $B_{\imath\jmath\ell}$.

Then by (\ref{LI}),

\begin{align}\label{B}
|\det B_{\imath\jmath\ell}|\geq \frac{(\gamma\langle\ell\rangle_d)^{(2b)^{\imath+\jmath+\ell}}}{2^{(2b)^{\imath+\jmath+\ell}}(|k|+1)^{\tau(2b)^{\imath+\jmath+\ell}}}.
\end{align}
Note that
\begin{align*}
|B_{\imath\jmath\ell}^{-1}|=|\frac{\textrm{adj} B_{\imath\jmath\ell}}{\det B_{\imath\jmath\ell}}|\leq c\frac{(|k|+1)^{\tau(2b)^{\imath+\jmath+\ell}+(2b)^{\imath+\jmath+\ell}-1}}{(\gamma\langle\ell\rangle_d)^{(2b)^{\imath+\jmath+\ell}}}.
\end{align*}
Applying the identity
\begin{align*}
\partial_y^jB_{\imath\jmath\ell}^{-1}=-\sum_{|j'|=1}^{|j|}\left(\begin{array}{c}
j\\
j'
\end{array}\right)(\partial_y^{j-j'}B_{\imath\jmath}^{-1}\partial_y^{j'}B_{\imath\jmath})B_{\imath\jmath}^{-1}
\end{align*}
inductively, we have
\begin{align}\label{B-}
|\partial_y^j B_{\imath\jmath\ell}^{-1}|_{ D(s)\times G_+}&\leq c(|k|+1)|^{|j|}|B_{\imath\jmath\ell}^{-1}|^{|j|+1}\\
&\leq c\frac{(1+|k|)^{|j|+(|j|+1)(2b)^{\imath+\jmath+\ell}\tau}}{(\gamma\langle\ell\rangle_d)^{(|j|+1)(2b)^{\imath+\jmath+\ell}}},~~2|j|\leq m.
\end{align}
Then
\begin{align*}
\|F_z\|_{D(s-\rho,r)}
&\leq \|\sum_{0<|k|\leq K_+,~2|\imath|+|\jmath|\leq m}B_{\imath\jmath\ell}^{-1}\partial_z(p_{k\imath\jmath\ell}z^\jmath)y^{\imath}w^{\ell_1}\bar w^{\ell_2}\|_{D(s-\rho,r)}\\
&\leq
\|\sum_{0<|k|\leq K_+,~2|\imath|+|\jmath|\leq m}\frac{(1+|k|)^{1+(2b)^{\imath+\jmath+\ell}\tau}}{(\gamma\langle\ell\rangle_d)^{(2b)^{\imath+\jmath+\ell}}}|\partial_z(p_{k\imath\jmath\ell}z^{\jmath})y^{\imath}w^{\ell_1}\bar w^{\ell_2}|{\rm e}^{|k|(s-\rho)}\|_{D(s-\rho,r)}\\
&\leq(\sum_{0<|k|\leq K_+,~2|\imath|+|\jmath|\leq m}(\frac{(1+|k|)^{1+(2b)^{\imath+\jmath+\ell}\tau}}{(\gamma\langle\ell\rangle_d)^{(2b)^{\imath+\jmath+\ell}}})^2{\rm e}^{-2|k|\rho})^{\frac{1}{2}}\\
&~~~~(\sum_{0<|k|\leq K_+,2|\imath|+|\jmath|\leq m}|\partial_z(p_{k\imath\jmath}z^{\jmath})y^\imath w^{\ell_1}\bar w^{\ell_2}|^2{\rm e}^{2|k|s})^{\frac{1}{2}}\\
&\leq \frac{A_\rho}{\gamma^{(2b)^{m+2}}}\|R_z\|_{D(s,r)},
\end{align*}
i.e.,
\begin{align}\label{Fz}
\frac{\|F_z\|_{D(s-\rho,r)}}{r^{a-1}}&\leq\frac{A_\rho}{r^{a-1}\gamma^{(2b)^{m+2}}}\|R_z\|_{D(s,r)}\leq\frac{A_\rho}{r^{a-1}}\gamma^{(2b)^{m+2}}r^{m-1}\eta^m.
\end{align}
To control the Lipschitz semi-norm of $F_z$. Let $\Delta=\Delta_{\xi\zeta}$ for $\xi$, $\zeta\in\Pi$. Note that
\begin{align*}
\Delta F_z&=\sum_{0<|k|\leq K_+,~2|\imath|+|\jmath|\leq m,~|\ell|\leq2}(\Delta\partial_z(F_{k\imath\jmath\ell}z^\jmath)y^\imath w^{\ell_1}\bar w^{\ell_2}\\
&=\sum_{0<|k|\leq K_+,~2|\imath|+|\jmath|\leq m,~|\ell|\leq2}\Delta B_{\imath\jmath\ell}^{-1}\partial_z(p_{k\imath\jmath\ell}z^\jmath)y^\imath w^{\ell_1}\bar w^{\ell_2}\\
&=\sum_{0<|k|\leq K_+,~2|\imath|+|\jmath|\leq m,~|\ell|\leq2}(B_{\imath\jmath\ell}^{-1}(\xi)\Delta \partial_z(p_{k\imath\jmath\ell}z^\jmath)+\Delta B_{\imath\jmath\ell}^{-1}\partial_z(p_{k\imath\jmath\ell}(\zeta)z^\jmath))y^\imath w^{\ell_1}\bar w^{\ell_2}{\rm e}^{\sqrt{-1}\langle k,x\rangle}\\
&=:U1+U2,
\end{align*}
where
\begin{align*}
U1&=\sum_{0<|k|\leq K_+,~2|\imath|+|\jmath|\leq m,~|\ell|\leq2}B_{\imath\jmath\ell}^{-1}(\xi)\Delta \partial_z(p_{k\imath\jmath\ell}z^\jmath) y^\imath w^{\ell_1}\bar w^{\ell_2}{\rm e}^{\sqrt{-1}\langle k,x\rangle},\\
U2&=\sum_{0<|k|\leq K_+,~2|\imath|+|\jmath|\leq m,~|\ell|\leq2}\Delta B_{\imath\jmath\ell}^{-1}\partial_z(p_{k\imath\jmath\ell}(\zeta)z^\jmath) y^\imath w^{\ell_1}\bar w^{\ell_2}{\rm e}^{\sqrt{-1}\langle k,x\rangle}.
\end{align*}
Notice by (\ref{B-}) that
\begin{align*}
\|U1\|_{D(s-\rho,r)}&\leq \sum_{0<|k|\leq K_+,~2|\imath|+|\jmath|\leq m,~|\ell|\leq2}\frac{(1+|k|)^{(2b)^{\imath+\jmath+\ell}\tau}}{(\gamma\langle\ell\rangle_d)^{(2b)^{\imath+\jmath+\ell}}}|\Delta \partial_z(p_{k\imath\jmath\ell}z^\jmath) y^\imath w^{\ell_1}\bar w^{\ell_2}|{\rm e}^{|k|(s-\rho)}\\
&\leq \frac{A_\rho}{\gamma^{(2b)^{m+2}}}\|\Delta R_z\|_{D(s,r)},
\end{align*}
and
\begin{align*}
\|U2\|_{D(s-\rho,r)}&\leq \sum_{0<|k|\leq K_+,~2|\imath|+|\jmath|\leq m,~|\ell|\leq2}M\frac{(1+|k|)^{(2b)^{\imath+\jmath+\ell}\tau}}{(\gamma\langle\ell\rangle_d)^{(2b)^{\imath+\jmath+\ell}}}\partial_z(p_{k\imath\jmath\ell}(\zeta)z^\jmath) y^\imath w^{\ell_1}\bar w^{\ell_2}{\rm e}^{|k|(s-\rho)}\\
&\leq \frac{MA_\rho}{\gamma^{(2b)^{m+2}}}\|R_z\|_{D(s,r)},
\end{align*}
where $M=|\omega|_{\Pi}^\mathcal{L}+\pmb|\Omega\pmb|_{-\delta,\Pi}^\mathcal{L}$.
Then
\begin{align*}
\|\Delta F_z\|_{D(s-\rho,r)}&\leq \frac{A_\rho}{\gamma^{(2b)^{m+2}}}(\|\Delta R_z\|_{D(s,r)}+{M}\|R_z\|_{D(s,r)}).
\end{align*}
Dividing by $|\xi-\zeta|$ and taking the supremum over $\xi\neq\zeta$ in $\Pi$ we arrive at
\begin{align}\label{FzL}
\frac{1}{r^{a-1}}\|F_z\|_{D(s-\rho,r)}^{\mathcal{L}}&\leq\frac{A_\rho}{\gamma^{(2b)^{m+2}}}(\pmb\|X_R\pmb\|^{\mathcal{L}}+{M}\pmb\|X_R\pmb\|_{D(s,r)})\\
&\leq \frac{MA_\rho}{\gamma^{(2b)^{m+2}}}(\gamma^{(2b)^{m+2}}r^{m-a+1}\eta^m+\gamma^{2(2b)^{m+2}}r^{m-a}\eta^m)\notag\\
&\leq \frac{MA_\rho}{\gamma^{(2b)^{m+2}}} \gamma^{(2b)^{m+2}}r^{m-a}\eta^m.\notag
\end{align}
Similarly, we have
\begin{align}\label{Fy}\frac{|F_y|_{D(s-\rho,r)}}{r^{a-2}}\leq cA_\rho r^{m-a}\eta^m,~~~~\frac{1}{r^{a-2}}|F_y|_{D(s-\rho,r)}^{\mathcal{L}}\leq c\frac{MA_\rho}{\gamma^{(2b)^{m+2}}} r^{m-a}\eta^m,
\end{align}
\begin{align}\label{Fx}\frac{|F_x|_{D(s-\rho,r)}}{r^a}\leq cA_\rho r^{m-a}\eta^m,~~~~\frac{1}{r^{a}}|F_x|_{D(s-\rho,r)}^{\mathcal{L}}\leq c\frac{MA_\rho}{\gamma^{(2b)^{m+2}}} r^{m-a}\eta^m.
\end{align}
Next we estimate $\|F_{w}\|_{\bar{p}}$, $\|F_{\bar w}\|_{\bar{p}}$.
Using the Lemma 1 in \cite{poschel2}, we have \begin{align}\label{Fw}\|F_{w}\|_{\bar{p}}\leq cA_\rho r^{m-a}\eta^m,~~~~\|F_{w}\|_{\bar{p}}^{\mathcal{L}}\leq c\frac{MA_\rho}{\gamma^{(2b)^{m+2}}} r^{m-a}\eta^m.
\end{align}
\begin{align}\label{Fbw}\|F_{\bar w}\|_{\bar{p}}\leq cA_\rho r^{m-a}\eta^m,~~~~\|F_{\bar w}\|_{\bar{p}}^{\mathcal{L}}\leq c\frac{MA_\rho}{\gamma^{(2b)^{m+2}}} r^{m-a}\eta^m.
\end{align}
Hence, in view of (\ref{Fz}), (\ref{FzL}),  (\ref{Fy}), (\ref{Fx}), (\ref{Fw}) and (\ref{Fbw}),
\begin{align*}
\pmb\|X_{F}\pmb\|_{r,D(s-\rho,{r})}+\frac{\gamma^{(2b)^{m+2}}}{M}\pmb\|X_F\pmb\|_{r,D(s-\rho,{r})}^{\mathcal{L}}
&\leq cA_\rho r^{m-a}\eta^m.
\end{align*}
{{By the generalized Cauchy estimate, we have
\begin{align*}
\pmb\|DX_F\pmb\|_{r,r,D(s-2\rho,\frac{r}{2})}<\frac{2^a}{\rho r^a}\pmb\|X_F\pmb\|_{r,D(s-\rho,r)},
\end{align*}
where on the left we use the operator norm
\begin{align*}
\pmb\|L\pmb\|_{r,r'}=\sup_{W\neq0}\frac{\pmb\|LW\pmb\|_{\bar p,r}}{\pmb\|W\pmb\|_{p,r'}}.
\end{align*}}}

The proof is complete.}

\end{proof}

Next we apply the above transformation $\phi_F^1$ to Hamiltonian $H$, i.e.,
\begin{align*}
H\circ\phi_F^1&=(N+R)\circ\phi_F^1+(P-R)\circ\phi_F^1\\
&=(N+R)+\{N,F\}+\int_0^1\{(1-t)\{N,F\}+ R,F\}\circ\phi_F^tdt+(P-R)\circ\phi_F^1\\
&=N+[R]+\int_0^1\{R_t,F\}\circ\phi_F^tdt+(P-R)\circ\phi_F^1+Q\\
&=:\bar N_++\bar P_+,
\end{align*}
and
\begin{align}\label{eq15}
\bar N_+&= N+[R]\notag\\
&=e+\langle\omega,y\rangle+\langle\Omega w,\bar w\rangle+g(z)+f(y,z,w,\bar w)+[R](y,z,w,\bar w),\\\label{eq16}
\bar P_+&=\int_0^1\{R_t,F\}\circ\phi_F^tdt+(P-R)\circ\phi_F^1+Q,\\
R_t&=(1-t)Q+(1-t)[R]+tR.\notag
\end{align}

\subsubsection{Translation}
In this subsection, we will eliminate the first order items of $z$.
Consider the translation
$$\phi:x\rightarrow x,~~~~~y\rightarrow y,~~~~~z\rightarrow z+\zeta_+-\zeta,$$
where $z=(w_0,\bar w_0)^\top$,  and $\zeta_+\in B_{ r}(\zeta)$ is to be determined.
Let
$$\Phi_+=\phi_F^1\circ\phi.$$
Then
\begin{align}
H\circ\Phi_+&=N_++P_+,\notag\\\label{eq17}
N_+&=\bar N_+\circ\phi,\\\label{eq18}
P_+&=\bar P_+\circ\phi,
\end{align}
with
\begin{align*}
N_+&=\bar{N}_+\circ\phi=(N+[R])\circ\phi\\
&=(e+\langle\omega,y\rangle+\langle\Omega w,\bar w\rangle+g(z)+f(y,z,w,\bar w)+[R](y,z,w,\bar w))\circ\phi\\
&=e+\langle\omega,y\rangle+\langle\Omega w,\bar w\rangle+g(z+\zeta_+-\zeta)+f(y,z+\zeta_+-\zeta,w,\bar w)+[R](y,z+\zeta_+-\zeta,w,\bar w)\\
&=:e_++\langle\omega_+,y\rangle+\langle\Omega_+ w,\bar w\rangle +g_++f_+,
\end{align*}
where
\begin{align}\label{e+}
e_+&=e+g(\zeta_+-\zeta)+[R](0,\zeta_+-\zeta,0,0),\\\label{omega+}
\omega_+&=\omega+\partial_y[R](0,0,0,0),\\\label{Omega+}
\Omega_+&=\Omega+\partial_{w,\bar w}[R](0,0,0,0),\\\label{g+}
g_+&={g(z+\zeta_+-\zeta)-g(\zeta_+-\zeta)+[R](0,z+\zeta_+-\zeta,0,0)-[R](0,\zeta_+-\zeta,0,0),}\\\label{f+}
f_+&=f(y,z+\zeta_+-\zeta,w,\bar w)+[R](y,z+\zeta_+-\zeta,w,\bar w)-[R](0,z+\zeta_+-\zeta,0,0)\\
&~~~~-\langle\partial_y[R](0,0,0,0),y\rangle-\langle\partial_{w,\bar w}[R](0,0,0,0)w,\bar w\rangle.\notag
\end{align}

\subsubsection{Eliminate the first order terms}\label{326}
In this subsection, we will appropriately choose $\zeta_+$ such that the first order terms about $z$ disappear. The concrete details see the following lemma.
\begin{lemma}\label{le3}
Let
\begin{align}\label{eq47}
\nabla g_+(0)=\nabla g(\zeta_+-\zeta)+\nabla_z[R](0,\zeta_+-\zeta,0,0).
\end{align}
There exists
$\zeta_+\in B_{ (r_-^{m-1}\eta_-^m)^\frac{1}{L}}(\zeta)$
such that
\begin{align*}
\nabla g_+(0)=\nabla g(0)=\cdots=\nabla g_0(0)=0.
\end{align*}
\end{lemma}
\begin{proof}
The proof will be completed by an induction on $\nu$. We start with the case $\nu=0$. It follows from $g(z)=o(\|z\|_{\textsf{a},\textsf{p}}^2)$  and \textbf{(A0)} that
\begin{align}\label{0}
&\nabla g_0(\zeta_0)=\nabla g_0(0)=0,~~\zeta_0\in O\\\label{degg0}
&\deg(\nabla g_0(\cdot)-\nabla g_0(0),O^o,0)\neq0,\\\label{nag0}
&\|\nabla g_0(z)-\nabla g_0(z_*)\|_{\bar{p}}\geq\sigma_0\|z-z_*\|_{{p}}^L,~~z,z_*\in O.
\end{align}
Now assume that for some $\nu\geq1$ we have got
\begin{align}\label{1}
&\nabla g_i(0)=\nabla g_{i-1}(0)=0,~~\zeta_i\in B_{ (r_{i-2}^{m-1}\eta^m_{i-2})^\frac{1}{L}}(\zeta_{i-1}),\\\label{degg}
&\deg(\nabla g_i(\cdot)-\nabla g_i(0),O^o,0)\neq0,\\\label{nag}
&\|\nabla g_i(z)-\nabla g_i(z_*)\|_{\bar{p}}\geq\sigma_i\|z-z_*\|_{{p}}^L, ~~z\in O\backslash B_{(r^{m-1}\eta^m)^{\frac{1}{L}}}(z_*), z_*\in O,
\end{align}
where $i=1,2,\cdots,\nu.$ Then we need to find $\zeta_+$ near $\zeta$ such that $\nabla g_+(0)=\nabla g(0)$.

Consider homotopy $H_t(z):[0,1]\times O\rightarrow \ell^{a,\bar{p}}\times\ell^{a,\bar{p}}$,
\begin{align*}
H_t(z)&=:\nabla g(z-\zeta)-\nabla g(0)+t\nabla_z[R](0,z-\zeta,0,0).
\end{align*}
Notice by (\ref{XR}) that
\begin{align}\label{ezr}
&\|\nabla_z[R](y,z,w,\bar w)\|_{\bar{p}}\leq r^{a-1}\pmb\|X_R\pmb\|_{r,D(s-\rho,8\eta r)}\leq\gamma^{2(2b)^{m+2}}r^{m-1}\eta^m.
\end{align}
For any $z\in\partial O$, $t\in[0,1]$, by (\ref{nag}) and (\ref{ezr}), we have
\begin{align*}
\|H_t(z)\|_{\bar{p}}
&\geq\|\nabla g(z-\zeta)-\nabla g(0)\|_{\bar{p}}-\|\nabla_z[R](0,z-\zeta,0,0)\|_{\bar{p}}\\
&\geq\sigma\|z-\zeta\|_{p}^L-\gamma^{2(2b)^{m+2}}r^{m-1}\eta^m\\
&>\frac{\sigma\delta^L}{2},
\end{align*}
where $\delta:=\min\{\|z-\zeta\|_{p}, \forall z\in\partial O\}$.
So, it follows from the homotopy invariance and (\ref{degg}) that
\begin{align}\label{H1}
\deg(H_1(\cdot),O^o,0)=\deg(H_0(\cdot),O^o,0)\neq0.
\end{align}
We note by (\ref{nag}) and (\ref{ezr}) that for any $z\in O\backslash B_{(r_{-}^{m-1}\eta_{-}^m)^{\frac{1}{L}}}(\zeta)$,
\begin{align*}
\|H_1(z)\|_{\bar{p}}&=\|\nabla g(z-\zeta)-\nabla g(0)+\nabla_z[R](0,z-\zeta,0,0)\|_{\bar{p}}\\
&\geq\|\nabla g(z-\zeta)-\nabla g(0)\|_{\bar{p}}-\|\nabla_z[R](0,z-\zeta,0,0)\|_{\bar{p}}\\
&\geq\sigma\|z-\zeta\|_p^L-\gamma_0^{2(2b)^{m+2}}r^{m-1}\eta^m\\
&\geq\sigma r_{-}^{m-1}\eta_{-}^m-\gamma_0^{2(2b)^{m+2}}r^{m-1}\eta^m\\
&\geq\frac{\sigma}{2}r_{-}^{m-1}\eta_{-}^m.
\end{align*}
Hence by excision and (\ref{H1}),
\begin{align*}
\deg(H_1(\cdot),B_{ (r_{-}^{m-1}\eta_{-}^m)^{\frac{1}{L}}}(\zeta),0)=\deg(H_1(\cdot),O^o,0)\neq0,
\end{align*}
then there exist at least a $\zeta_+\in B_{ (r_{-}^{m-1}\eta_{-}^m)^{\frac{1}{L}}}(\zeta)$, such that
\begin{align*}
H_1(\zeta_+)=0,
\end{align*}
i.e.,
\begin{align*}
\nabla g(\zeta_+-\zeta)+\nabla_z[R](0,\zeta_+-\zeta,0,0)=\nabla g(0),
\end{align*}
thus
\begin{align}\label{g+=g0}
\nabla g_+(0)=\nabla g(0)=\cdots=\nabla g_0(0)=0.
\end{align}

Next we need to prove
\begin{align}\label{degg+}
&\deg(\nabla g_+(\cdot)-\nabla g_+(0),O^o,0)\neq0,\\\label{nag+}
&\|\nabla g_+(z)-\nabla g_+(z_*)\|_{\bar{p}}\geq\sigma_+\|z-z_*\|_{{p}}^L.
\end{align}
By (\ref{g+}),
\begin{align*}
\nabla g_+(z)=\nabla g(z+\zeta_+-\zeta)+\nabla[R](0,z+\zeta_+-\zeta,0,0).
\end{align*}
Then
\begin{align}\label{g+-g}
\nabla g_+(z)-\nabla g(z)&=\nabla g(z+\zeta_+-\zeta)-\nabla g(z)+\nabla[R](0,z+\zeta_+-\zeta,0,0),
\end{align}
and
\begin{align}\label{g+-g+}
\nabla g_+(z)-\nabla g_+(z_*)&=\nabla g(z+\zeta_+-\zeta)-\nabla g(z_*+\zeta_+-\zeta)+\nabla[R](0,z+\zeta_+-\zeta,0,0)\\
&~~~-\nabla[R](0,z_*+\zeta_+-\zeta,0,0).\notag
\end{align}
In view of (\ref{ezr}), (\ref{g+-g}), and $\zeta_+\in B_{(r_{-}^{m-1}\eta_{-}^m)^\frac{1}{L}}(\zeta)$, we get
\begin{align}\label{nag+-nag}
\|\nabla g_+(z)-\nabla g(z)\|_{\bar{p}}\leq c(r_{-}^{m-1}\eta_{-}^m)^\frac{1}{L},
\end{align}
so, it follows from the property of degree, (\ref{degg}) and (\ref{g+=g0}) that (\ref{degg+}) holds, i.e.,
\begin{align*}
\deg(\nabla g_+(\cdot)-\nabla g_+(0),O^o,0)=\deg(\nabla g_+(\cdot)-\nabla g(\cdot)+\nabla g(\cdot)-\nabla g(0),O^o,0)\neq0.
\end{align*}
According to (\ref{nag}), (\ref{ezr}) and (\ref{g+-g+}), we have for any $z\in O\backslash B_{(r^{m-1}\eta^m)^{\frac{1}{L}}}(z_*)$,
\begin{align*}
\|\nabla g_+(z)-\nabla g_+(z_*)\|_{\bar{p}}\geq\sigma\|z-z_*\|_{{p}}^L-4\gamma_0^{2(2b)^{m+2}}r^{m-1}\eta^m\geq\sigma_+\|z-z_*\|_{{p}}^L,
\end{align*}
which implies (\ref{nag+}).

The proof is complete.
\end{proof}

\subsubsection{Frequency Property}
In view of (\ref{omega+}), (\ref{Omega+}) and $\pmb\|X_{[R]}\pmb\|\leq c\pmb\|X_{P}\pmb\|$, we can get $|\omega_+-\omega|<c\pmb\|X_P\pmb\|_r$ and $\|(\Omega_+-\Omega)w\|_{\bar p}<cr^a\pmb\|X_P\pmb\|_r$ on $D(s,r)$, hence $\pmb|\Omega_+-\Omega\pmb|_{\bar p-p}<c\pmb\|X_P\pmb\|_r$ on $\Pi$. The same holds for their Lipschitz semi-norms with $-\delta\leq\bar p-p$, and we get
\begin{align}\label{om+-om}
|\omega_+-\omega|+\pmb|\Omega_+-\Omega\pmb|_{-\delta}<c\pmb\|X_P\pmb\|_r,~~|\omega_+-\omega|^{\mathcal{L}}+\pmb|\Omega_+-\Omega\pmb|_{-\delta}^{\mathcal{L}}<c\pmb\|X_P\pmb\|_r^{\mathcal{L}}.
\end{align}
In order to bound the small divisors for the new frequencies $\omega_+$ and $\Omega_+$ for $|k|<K_+$, we observe that $|\ell|_\delta\leq|\ell|_{d-1}\leq 2\langle\ell\rangle_d$, hence
$$|\langle k,\omega_+-\omega\rangle+\langle\ell,\Omega_+-\Omega\rangle|\leq|k||\omega_+-\omega|+|\ell|_\delta\pmb|\Omega_+-\Omega\pmb|_{-\delta}<K_+\langle\ell\rangle_d\pmb\|X_P\pmb\|_r\leq(\gamma-\gamma_+)\frac{\langle\ell\rangle_d}{(1+|k|)^\tau},$$
where $\gamma-\gamma_+>cK_+\max_{|k|\leq K_+}(1+|k|)^\tau\pmb\|X_P\pmb\|_r$.
The new ones then satisfy
\begin{align*}
|\langle k,\omega_+\rangle+\langle \ell,\Omega_+ \rangle|\geq \gamma_+\frac{\langle\ell\rangle_d}{A_k}
\end{align*}
on $\Pi$.

\subsubsection{Estimate on $\Phi_+$}
\begin{lemma}\label{le7}
In addition to \textbf{\textsc{(H1)}}-\textbf{\textsc{(H2)}}. Assume that
\begin{align*}
&\textbf{\textsc{(H3)}}~{A_\rho}r^{m-1}\eta^{m}<\rho,\\
&\textbf{\textsc{(H4)}}~A_\rho r^{m-2a}\eta^{m-a}<1.
\end{align*}

Then the following hold.
\begin{itemize}
\item[(1)]For all $0\leq t\leq 1$,
\begin{align}\label{eq26}
\phi_F^t&:D(s-5\rho,2\eta r)\rightarrow D(s-4\rho,4\eta r),\\\label{eq27}
\phi&:D(s-6\rho,\eta r)\rightarrow D(s-5\rho,2\eta r),
\end{align}
are well defined.
\item[(2)]$\Phi_+:D_+=D(s_+,r_+)=D(s-6\rho,\eta r)\rightarrow D(s,r).$
\item[(3)]There is a constant $c_3$ such that
\begin{align*}
\pmb\|\phi_F^t-id\pmb\|^{*}_{r,D(s-5\rho,2\eta r)}&\leq c_3\pmb\|X_F\pmb\|_{r,D(s-2\rho,4\eta r)}^*\leq c_3A_\rho r^{m-a}\eta^m,\\
\pmb\|D\phi_F^t-Id\pmb\|^{*}_{r,r,D(s-6\rho,\eta r)}&\leq c_3\frac{1}{\rho r^a}\pmb\|X_F\pmb\|_{r,D(s-2\rho,4\eta r)}^*\leq c_3\frac{A_\rho}{\rho}r^{m-2a}\eta^{m-a}.
\end{align*}
\item[(4)]
\begin{align*}
\pmb\|\Phi_+-id\pmb\|^{*}_{r,D(s-5\rho,2\eta r)}&\leq c_3\pmb\|X_F\pmb\|_{r,D(s-6\rho,4\eta r)}^*\leq c_3A_\rho r^{m-a}\eta^m,\\
\pmb\|D\Phi_+-Id\pmb\|^{*}_{r,r,D(s-4\rho,\eta r)}&\leq c_3\frac{1}{\rho r^a}\pmb\|X_F\pmb\|_{r,D(s-2\rho,4\eta r)}^*\leq c_3\frac{A_\rho}{\rho}r^{m-2a}\eta^{m-a}.
\end{align*}
\end{itemize}
\end{lemma}
\begin{proof}
(1)~~(\ref{eq27}) immediately follows from $\zeta_+\in B_{(r_{-}^{m-1}\eta_{-}^m)^\frac{1}{L}}(\zeta)$ in Lemma \ref{le3}. Indeed, for $\forall x,y,z,w,\bar w, \in D(s-6\rho,\eta r)$,  $\phi(x,y,z,w,\bar w)=(x,y,z+\zeta_+-\zeta,w,\bar w)$, then
\begin{align*}
||z+\zeta_+-\zeta||_p&<\eta r+c(r_{-}^{m-1}\eta_{-}^m)^\frac{1}{L}<2\eta r,
\end{align*}
as $m\geq L+\frac{\sqrt{4L^2+2L}}{2}$.
To verify (\ref{eq26}), we denote $\phi_{F_1}^t$, $\phi_{F_2}^t$, $\phi_{F_3}^t$, $\phi_{F_4}^t$, $\phi_{F_5}^t$ as the components of $\phi_{F}^t$ in the $x$, $y$, $z$, $w$, $\bar w$ planes, respectively. Then
\begin{align}\label{eq28}
\phi_F^t=id+\int_0^tX_F\circ \phi_F^\lambda d\lambda,~~~~0\leq t\leq 1.
\end{align}
For any $(x,y,z,w,\bar w)\in D(s-5\rho,2\eta r)$, we let $t_*=\sup\{t\in[0,1]:\phi_F^t(x,y,z,w,\bar w)\in D(s-4\rho,4\eta r)\}$. Then for any $0\leq t\leq t_*$, by the definition of $\pmb\|\cdot\pmb\|_{r,D(s,r)}$, (\ref{XP}), (\ref{XR}), (\ref{XF}), (\ref{eq28}), \textbf{(H3)} and \textbf{(H4)}, we get
\begin{align*}
|\phi_{F_1}^t(x,y,z)|_{D(s-5\rho,2\eta r)}&\leq|x|_{D(s-5\rho,2\eta r)}+\int_0^t|F_y\circ\phi_F^\lambda|_{D(s-5\rho,2\eta r)}d\lambda\\
&\leq s-5\rho+r^{a-1}\pmb\|X_F\pmb\|_{r,D(s-4\rho,4\eta r)}\\
&<s-5\rho+cA_\rho r^{m-1}\eta^m\\
&\leq s-4\rho,\\
|\phi_{F_2}^t(x,y,z)|_{D(s-5\rho,2\eta r)}&\leq|y|_{D(s-5\rho,2\eta r)}+\int_0^t|-F_x\circ\phi_F^\lambda|_{D(s-5\rho,2\eta r)}d\lambda\\
&\leq (2\eta r)^2+r^a\pmb\|X_F\pmb\|_{r,D(s-4\rho,4\eta r)}\\
&<(2\eta r)^2+c A_\rho r^m\eta^m\\
&<(4\eta r)^2,\\
|\phi_{F_3}^t(x,y,z)|_{D(s-5\rho,2\eta r)}&\leq|z|_{D(s-5\rho,2\eta r)}+\int_0^t|\tilde{J}F_z\circ\phi_F^\lambda|_{D(s-5\rho,2\eta r)}d\lambda\\
&\leq 2\eta r+r\pmb\|X_F\pmb\|_{r,D(s-4\rho,4\eta r)}\\
&\leq 2\eta r+cA_\rho r^{m-a+1}\eta^m\\
&<4\eta r,\\
|\phi_{F_4}^t(x,y,z)|_{D(s-5\rho,2\eta r)}&\leq|w|_{D(s-5\rho,2\eta r)}+\int_0^t|iF_{\bar w}\circ\phi_F^\lambda|_{D(s-5\rho,2\eta r)}d\lambda\\
&\leq (2\eta r)^a+\pmb\|X_F\pmb\|_{r,D(s-4\rho,4\eta r)}\\
&\leq(2\eta r)^a+cA_\rho r^{m-a}\eta^m\\
&<(4\eta r)^a,\\
|\phi_{F_5}^t(x,y,z)|_{D(s-5\rho,2\eta r)}&\leq|\bar w|_{D(s-5\rho,2\eta r)}+\int_0^t|-iF_{ w}\circ\phi_F^\lambda|_{D(s-5\rho,2\eta r)}d\lambda\\
&\leq (2\eta r)^a+\pmb\|X_F\pmb\|_{r,D(s-4\rho,4\eta r)}\\
&\leq(2\eta r)^a+cA_\rho r^{m-a}\eta^m\\
&<(4\eta r)^a.
\end{align*}
Thus, $\phi_F^t\in D(s-4\rho,4\eta r)\subset D(s,r)$, i.e. $t_*=1$ and (1) holds.

(2)~~It follows from (1) that (2) holds.

(3)~~By (\ref{XP}), (\ref{XR}), (\ref{XF}) and (\ref{eq28}),
\begin{align*}
\pmb\|\phi_F^t-id\pmb\|_{r,D(s-5\rho,2\eta r)}&=\pmb\|\int_0^tX_F\circ \phi_F^\lambda d\lambda\pmb\|_{r,D(s-5\rho,2\eta r)}\leq c\pmb\|X_F\pmb\|_{r,D(s-4\rho,4\eta r)}.
\end{align*}

Using Lemma A.4 in \cite{poschel2}, we have
$$\pmb\|\phi_F^t-id\pmb\|^{\mathcal{L}}_{r,D(s-5\rho,2\eta r)}\leq exp(\pmb\|DX_F\pmb\|_{D(s-4\rho,4\eta r)})\pmb\|X_F\pmb\|^{\mathcal{L}}_{D(s-4\rho,4\eta r)}\leq c_4\pmb\|X_F\pmb\|_{D(s-4\rho,4\eta r)}^{\mathcal{L}}.$$

{{By the generalized Cauchy estimate,
\begin{align}\label{Dphi-I}
\pmb\|D\phi_F^t-I\pmb\|^{*}_{r,r,D(s-6\rho,\eta r)}&<\frac{\pmb\|\phi_F^t-id\pmb\|^{*}_{r,D(s-5\rho,2\eta r)}}{\rho(\eta r)^a}<c\frac{\pmb\|X_F\pmb\|^*_{D(s-4\rho,4\eta r)}}{\rho(\eta r)^a}=c\frac{A_\rho}{\rho}r^{m-2a}\eta^{m-a}.
\end{align}}}
%

(4) now follows from (3).
\end{proof}

\subsubsection{Estimate on $P_+$}
In the following, we estimate the next step $P_+$.
\begin{lemma}\label{le8}
Assume $\textbf{\textsc{(H1)}}$-$\textbf{\textsc{(H4)}}$ and
\begin{align*}
&\textbf{\textsc{(H5)}}~ \Delta=:(\frac{A_\rho^2}{\rho^2}r^{m-2a}\eta^{-\frac{1}{2}}
+c\eta^{\frac{1}{2}}\gamma^{(2b)^{m+2}}+c\frac{A_{\rho}}{\rho}\eta^{\frac{3}{2}})\leq\gamma_+^{2(2b)^{m+2}}.
\end{align*}

Then
\begin{equation}\label{eXP+}
\pmb\|X_{P_+}\pmb\|^{\lambda}_{r_+,D(s_+,r_+)}\leq \gamma_+^{2(2b)^{m+2}}r_+^{m-a}\eta_+^m,
\end{equation}
where $0<\lambda<\frac{\gamma^{(2b)^{m+2}}}{M}$.
\end{lemma}
\begin{proof}
Recall (\ref{eq16}) and (\ref{eq18}), i.e.,
\begin{align*}
P_+&=\int_0^1\{R_t,F\}\circ\Phi_+dt+(P-R)\circ\Phi_++Q\circ\phi,\\
R_t&=(1-t)Q+(1-t)[R]+tR.
\end{align*}
Hence the new perturbing vectorfield is
\begin{align}\label{XP+}
X_{P_+}=\int_{0}^1(\Phi_+)^*[X_{R_t},X_{F}]dt+(\Phi_+)^*(X_P-X_R)+\phi^*X_{Q},
\end{align}
where $(\Phi_+)^*(X_P-X_R)=D\Phi_+^{-1}(X_P-X_R)\circ\Phi_+$,  $[X_{R_t},X_{F}]=JD\nabla R_tX_F-JD\nabla FX_{R_t}$.

Let $\Phi_{+}$ map $D(s-6\rho,\eta r)$ into $D(s-4\rho, 4\eta r)$. Using the method in \cite{poschel2} on page of $14$-$15$, we can prove that for $0<\lambda<\frac{\gamma^{(2b)^{m+2}}}{M}$
\begin{align}\label{Phi*}
\pmb\|\Phi_{+}^*Y\pmb\|_{{\eta r},D(s-6\rho,\eta r)}&<c\pmb\|Y\pmb\|_{\eta r,D(s-4\rho,4\eta r)},\\\label{phi*l}
\pmb\|\Phi_{+}^*Y\pmb\|^{\lambda}_{{\eta r},D(s-6\rho,\eta r)}&<c\pmb\|Y\pmb\|^{\lambda}_{\eta r,D(s-3\rho,5\eta r)}.
\end{align}
Then in view of (\ref{XP-R}) and (\ref{Phi*}), we can prove
\begin{align}\label{Phi*XP}
\pmb\|(\Phi_+)^*(X_P-X_R)\pmb\|^{\lambda}_{\eta r,D(s-6\rho,\eta r)}\leq c\pmb\|X_P-X_R\pmb\|^{\lambda}_{\eta r,D(s-4\rho,4\eta r)}\leq c\eta^{m+1-a}\pmb\|X_P\pmb\|^{\lambda}_{ r,D(s,r)}.
\end{align}

Recall the definition of $Q$ in (\ref{Q}), i.e.,
\begin{align*}
Q=(\partial_zg+\partial_zf)J\partial_zF|_{2|\imath|+|\jmath|> m, or |\ell|\geq 2}.
\end{align*}
In the following, we estimate
\begin{align*}
\pmb\|X_Q\pmb\|_{\eta r,D(s-2\rho,6\eta r)}=\sup_{D(s-2\rho,6\eta r)}\{\frac{|Q_y|}{(\eta r)^{a-2}}+\frac{|Q_x|}{(\eta r)^a}+\frac{||JQ_z||_{\bar p}}{(\eta r)^{a-1}}+\|Q_{w}\|+\|Q_{\bar w}\|\}.
\end{align*}
We calculate
\begin{align}\label{Qy}
\frac{|Q_y|_{D(s-2\rho,6\eta r)}}{(\eta r)^{a-2}}&<\frac{|Q|_{D(s-2\rho,8\eta r)}}{(\eta r)^a}<\frac{|(\partial_zg+\partial_zf)I\partial_zF|_{D(s-2\rho,8\eta r)}}{(\eta r)^a}<c\frac{r^{a-1}\pmb\|X_F\pmb\|_{r,{D(s-2\rho,8\eta r)}}}{(\eta r)^{a-2}}\\
&<cA_\rho r^{m-a+1}\eta^{m-a+2},\notag
\end{align}
and
\begin{align}\label{QyL}
\frac{|Q_y|^{\mathcal{L}}_{D(s-2\rho,6\eta r)}}{(\eta r)^{a-2}}&<cr^{a-1}\pmb\|X_F\pmb\|+cr^{a-1}\pmb\|X_F\pmb\|^{\mathcal{L}}<c\frac{M}{\gamma^{(2b)^{2\imath+\jmath+\ell}}}A_\rho r^{m-a+1}\eta^{m-a+2}.
\end{align}
Similarly, we can prove
\begin{align}\label{Qz}
\frac{|Q_x|_{D(s-2\rho,6\eta r)}}{(\eta r)^a}, \frac{||Q_z||_{\bar p,D(s-2\rho,6\eta r)}}{(\eta r)^{a-1}},||Q_w||_{\bar p,D(s-2\rho,6\eta r)},||Q_{\bar w}||_{\bar p,D(s-2\rho,6\eta r)}\leq c\frac{A_\rho}{\rho}r^{m-a+1}\eta^{m-a+2},
\end{align}
and
\begin{align}
\frac{|Q_x|^{\mathcal{L}}_{D(s-2\rho,6\eta r)}}{(\eta r)^a}, \frac{||Q_z||^{\mathcal{L}}_{\bar p,D(s-2\rho,6\eta r)}}{(\eta r)^{a-1}},||Q_w||^{\mathcal{L}}_{\bar p,D(s-2\rho,6\eta r)},||Q_{\bar w}||^{\mathcal{L}}_{\bar p,D(s-2\rho,6\eta r)}\leq c\frac{A_\rho}{\rho}\frac{M}{\gamma^{(2b)^{2\imath+\jmath+\ell}}}r^{m-a+1}\eta^{m-a+2}.
\end{align}
Then
\begin{align}\label{XQ}
\pmb\|X_Q\pmb\|^{\lambda}_{\eta r,D(s-2\rho,6\eta r)}\leq c\frac{A_{\rho}}{\rho}r^{m-a+1}\eta^{m-a+2}.
\end{align}
By (\ref{XR}) and (\ref{XQ}), we can check that
\begin{align}\label{XRt}
\pmb\|X_{R_t}\pmb\|_{\eta r,D(s-2\rho,6\eta r)}&\leq \pmb\|X_{Q}\pmb\|_{\eta r,D(s-2\rho,6\eta r)}+\pmb\|X_{[R]}\pmb\|_{\eta r,D(s-\rho,8\eta r)}+\pmb\|X_{R}\pmb\|_{\eta r,D(s-\rho,8\eta r)}\notag\\
&\leq c\frac{A_{\rho}}{\rho}r^{m-a+1}\eta^{m-a+2}+c\gamma^{2(2b)^{m+2}}r^{m-a}\eta^{m},\notag\\
&\leq c\frac{A_{\rho}}{\rho}r^{m-a}\eta^{m},
\end{align}
and
\begin{align}
\pmb\|X_{R_t}\pmb\|^\mathcal{L}_{\eta r,D(s-2\rho,6\eta r)}&\leq c\frac{A_\rho}{\rho}\frac{M}{\gamma^{(2b)^{2\imath+\jmath+\ell}}}r^{m-a}\eta^{m}.
\end{align}
Using the generalized Cauchy estimate, Lemmas \ref{le1} and \ref{le2}, we get
\begin{align}\label{DXF}
\pmb\|DX_F\pmb\|^*_{r,r,D(s-3\rho, 5\eta r)}&\leq\frac{\pmb\|X_F\pmb\|^*_{r,D(s-2\rho,6\eta r)}}{\rho(\eta r)^a},~~~
\pmb\|DX_{R_t}\pmb\|^*_{r,r,D(s-3\rho, 5\eta r)}\leq\frac{\pmb\|X_{R_t}\pmb\|^*_{r,D(s-2\rho,6\eta r)}}{\rho(\eta r)^a}.
\end{align}
Then (\ref{XR}), (\ref{XF}), (\ref{XRt}), (\ref{DXF}) together with the definition of $[\cdot,\cdot]$ yield
\begin{align}\label{[XRt]}
\pmb\|[X_{R_t},X_F]\pmb\|_{\eta r,D(s-3\rho, 5\eta r)}&\leq\pmb\|DX_{R_t}\cdot X_F\pmb\|_{\eta r,D(s-3\rho, 5\eta r)}+\pmb\|DX_F\cdot X_{R_t} \pmb\|_{\eta r,D(s-3\rho, 5\eta r)}\\
&\leq\pmb\|DX_{R_t}\pmb\|_{\eta r,\eta r}\pmb\|X_F\pmb\|_{\eta r,D(s-3\rho, 5\eta r)}+\pmb\|DX_F \pmb\|_{\eta r,\eta r}\pmb\|X_{R_t} \pmb\|_{\eta r,D(s-3\rho, 5\eta r)}\notag\\
&\leq\frac{\pmb\|X_F\pmb\|_{\eta r,D(s-2\rho, 6\eta r)}\pmb\|X_{R_t} \pmb\|_{\eta r,D(s-2\rho, 6\eta r)}}{\rho (\eta r)^a}\notag\\
&\leq c\frac{A_\rho^2}{\rho^2}r^{2m-3a}\eta^{2m-a}.\notag
\end{align}
Similarly,
\begin{align}\label{[XRtL]}
\pmb\|[X_{R_t},X_F]\pmb\|^{\mathcal{L}}_{\eta r,D(s-3\rho, 5\eta r)}&\leq\pmb\|DX_{R_t}\cdot X_F\pmb\|^{\mathcal{L}}_{\eta r,D(s-3\rho, 5\eta r)}+\pmb\|DX_F\cdot X_{R_t} \pmb\|^{\mathcal{L}}_{\eta r,D(s-3\rho, 5\eta r)}\\
&\leq\pmb\|DX_{R_t}\pmb\|^{\mathcal{L}}_{\eta r,\eta r}\pmb\|X_F\pmb\|_{\eta r,D(s-3\rho, 5\eta r)}+\pmb\|DX_{R_t}\pmb\|_{\eta r,\eta r}\pmb\|X_F\pmb\|^{\mathcal{L}}_{\eta r,D(s-3\rho, 5\eta r)}\notag\\
&~~~~+\pmb\|DX_F \pmb\|^{\mathcal{L}}_{\eta r,\eta r}\pmb\|X_{R_t} \pmb\|_{\eta r,D(s-3\rho, 5\eta r)}+\pmb\|DX_F \pmb\|_{\eta r,\eta r}\pmb\|X_{R_t} \pmb\|^{\mathcal{L}}_{\eta r,D(s-3\rho, 5\eta r)}\notag\\
&\leq c\frac{A_\rho^2}{\rho^2}\frac{M}{\gamma^{(2b)^{m+2}}}r^{2m-3a}\eta^{2m-a}.\notag
\end{align}
So, by (\ref{phi*l}), (\ref{[XRt]}) and (\ref{[XRtL]}), we have
\begin{align}\label{Phi*XRt}
\pmb\|(\Phi_+)^*[X_{R_t},X_{F}]\pmb\|^{\lambda}_{\eta r,D(s-6\rho,\eta r)}&\leq c\pmb\|[X_{R_t},X_{F}]\pmb\|^{\lambda}_{\eta r,D(s-3\rho,5\eta r)}\leq c\frac{A_\rho^2}{\rho^2}r^{2m-3a}\eta^{2m-a}. 
\end{align}

Recall (\ref{XP+}) and collect all terms (\ref{Phi*XP}), (\ref{XQ}), (\ref{Phi*XRt}) we then arrive at the estimate
\begin{align*}
\pmb\|X_{P_+}\pmb\|^{\lambda}_{\eta r,D(s-5\rho,\eta r)}&\leq c\frac{A_\rho^2}{\rho^2}r^{2m-3a}\eta^{2m-a}
+c\eta^{m-a+1}\gamma^{(2b)^{m+2}}r^{m-a}\eta^m+c\frac{A_{\rho}}{\rho}r^{m-a+1}\eta^{m-a+2}\\
&\leq cr_+^{m-a}\eta_+^m(\frac{A_\rho^2}{\rho^2}r^{m-2a}\eta^{-\frac{1}{2}}
+c\eta^{\frac{1}{2}}\gamma^{(2b)^{m+2}}+c\frac{A_{\rho}}{\rho}\eta^{\frac{3}{2}})\\
&\leq\Delta r_+^{m-a}\eta_+^m\\
&\leq \gamma_+^{2(2b)^{m+2}}r_+^{m-a}\eta_+^m,
\end{align*}
where the last inequality follows from \textbf{\textsc{(H5)}}.

The proof is complete.

\end{proof}
This completes one cycle of KAM steps.
\section{Proof of Main Results}\label{sec:proof}
\subsection{Iteration Lemma}
In this section, we will prove an iteration lemma which guarantees the inductive construction of the transformations in all KAM steps.

Let $r_0,s_0,\gamma_0,
\eta_0,H_0,N_0,P_0$ be given at the beginning of Section \ref{sec:KAM} and let 
$D_0=D(s_0,r_0)$, $K_0=0$, $\Phi_0=id$. We define the following sequence inductively for all $\nu=1,2,\cdots$:
\begin{align*}
\rho_{\nu}&=\frac{\rho_0}{2^\nu},\\
\eta_\nu&=\eta_{\nu-1}^{1+\frac{1}{2m}},\\
r_\nu&=\eta_{\nu-1}r_{\nu-1},\\
s_\nu&=s_{\nu-1}-6\rho_{\nu-1},\\
\gamma_\nu&=\gamma_0(\frac{1}{2}+\frac{1}{2^\nu}),\\
\sigma_\nu&=\sigma_0(\frac{1}{2}+\frac{1}{2^\nu}),\\
M_\nu&=M_0(2-\frac{1}{2^\nu}),\\
K_\nu&=([\log(\frac{1}{\eta_{\nu-1}^{m+1}})]+1)^{3\mu},\\
D_\nu&=D(s_\nu, r_\nu),\\
{\Pi_\nu}&=\{\xi\in\Pi_{\nu-1}:|\langle k,\omega_{\nu-1}(\xi)\rangle+\langle\ell,\Omega_{\nu-1}(\xi)\rangle|\geq\frac{\gamma_{\nu-1}\langle \ell\rangle_d}{(1+|k|)^\tau}, \\
&~~~~~~|k|\leq K_+, |\ell|\leq2, |k|+|\ell|\neq0 \}.
\end{align*}
\begin{lemma}\label{le9}
Denote
\begin{align*}
\eta_*^2=\eta_0^m.
\end{align*}
If $\varepsilon$ is small enough, then the KAM step described in Section \ref{sec:KAM} is valid for all $\nu=0,1,\cdots$, resulting the sequences
$$e_\nu, \omega_\nu, \Omega_\nu, g_\nu, f_\nu, P_\nu, \Phi_\nu, H_\nu$$
$\nu=1,2,\cdots,$ with the following properties:
\begin{itemize}
\item[(1)]
\begin{align}\label{eq30}
|{e_{\nu+1}}-{e_{\nu}}|_{\Pi_\nu}&\leq \frac{\eta_*^\frac{1}{2}}{2^{\nu-1}},\\\label{eq31}
|{e_{\nu+1}}-{e_{0}}|_{\Pi_\nu}&\leq2\eta_*^{\frac{1}{2}},\\\label{eq32}
|{\omega_{\nu+1}}-{\omega_{\nu}}|^{\lambda_\nu}_{\Pi_\nu}&\leq \frac{\eta_*^\frac{1}{2}}{2^\nu},\\\label{eq33}
|{\omega_{\nu+1}}-{\omega_{0}}|^{\lambda_\nu}_{\Pi_\nu}&\leq2\eta_*^{\frac{1}{2}},\\\label{eq34}
\pmb|{\Omega_{\nu+1}}-{\Omega_{\nu}}\pmb|^{\lambda_\nu}_{-\delta,\Pi_\nu}&\leq \frac{\eta_*^\frac{1}{2}}{2^\nu},\\\label{eq35}
\pmb|{\Omega_{\nu+1}}-{\Omega_{0}}\pmb|^{\lambda_\nu}_{-\delta,\Pi_\nu}&\leq2\eta_*^{\frac{1}{2}},\\\label{eq36}
|{g_{\nu+1}}-{g_{\nu}}|_{D(s_\nu,r_{\nu})}&\leq \frac{\eta_*^\frac{1}{2}}{2^{\nu-1}},\\\label{eq37}
|{g_{\nu+1}}-{g_{0}}|_{D(s_\nu,r_{\nu})}&\leq2\eta_*^{\frac{1}{2}},\\\label{eq38}
|{f_{\nu+1}}-{f_{\nu}}|_{D(s_\nu,r_{\nu})}&\leq \frac{\eta_*^\frac{1}{2}}{2^{\nu-1}},\\\label{eq39}
|{f_{\nu+1}}-{f_{0}}|_{D(s_\nu,r_{\nu})}&\leq 2\eta_*^{\frac{1}{2}},\\\label{eq40}
\pmb\|X_{P_\nu}\pmb\|^{\lambda_\nu}_{r_\nu,D(s_\nu,r_{\nu})}&\leq\frac{\eta_*^\frac{1}{2}}{2^\nu},\\\label{eq41}
||\zeta_{\nu+1}-\zeta_{\nu}||_{p}&\leq \frac{\eta_*^\frac{1}{2}}{2^\nu}.
\end{align}
\item[(2)] There exist a Lipschitz family of real analytic symplectic coordinate transformations  $\Phi_{\nu+1}:{D}_{\nu+1}\times \Pi_{\nu+1}\rightarrow {D}_{\nu}$ and a closed subset $$\Pi_{\nu+1}=\Pi_\nu\setminus\cup_{|k|>K_\nu}\mathcal{R}_{kl}^{\nu+1}(\gamma_{\nu+1})$$
of $\Pi_\nu$, where $$\mathcal{R}_{kl}^{\nu+1}(\gamma_{\nu+1})=\{\xi\in\Pi_\nu:|\langle k,\omega_{\nu+1}\rangle+\langle\ell,\Omega_{\nu+1}\rangle|<\frac{\gamma_{\nu+1}\langle\ell\rangle_d}{(1+|k|)^\tau}\},$$
such that on $D_{\nu+1}\times\Pi_{\nu+1}$, \begin{equation*}
H_{\nu+1}=H_\nu\circ\Phi_{\nu+1}=N_{\nu+1}+P_{\nu+1},
\end{equation*}
\begin{align}\label{Peq36}
||\Phi_{\nu+1}-id||_{{D}_{\nu+1}}\leq\frac{\eta_*^\frac{1}{2}}{2^\nu},
\end{align}
and the same estimates as above are satisfied with $\nu+1$ in place of $\nu$, that is,
\begin{align} \label{M+} |\omega_{\nu+1}(\xi)|_{\Pi_{\nu+1}}^\mathcal{L}+\pmb|\Omega_{\nu+1}(\xi)\pmb|_{-\delta,\Pi_{\nu+1}}^\mathcal{L}\leq M_{\nu+1},\end{align}
\begin{align}\label{XP+2}\pmb\|X_{P_{\nu+1}}\pmb\|^{\lambda_{\nu+1}}_{D(s_{\nu+1},r_{\nu+1}),\Pi_{\nu+1}}<\gamma_{\nu+1}^{2(2b)^{m+2}}r_{\nu+1}^{m-a}\eta_{\nu+1}^m,\end{align}
and \begin{align}\label{Pi+}|\Pi_{\nu+1}\setminus\Pi_\nu|<c\gamma_0\frac{1}{1+K_{\nu-1}}.\end{align}
\end{itemize}
\end{lemma}
\begin{proof}
The proof amounts to the verification of $\textbf{(H1)}$-$\textbf{(H5)}$ for all $\nu$.
According to the definition of $r_\nu$ and $\eta_\nu$, we note that
\begin{align*}
r_\nu=\eta_0^{2m((1+\frac{1}{2m})^\nu-1)+m},~~~\eta_\nu=\eta_0^{(1+\frac{1}{2m})^\nu}.
\end{align*}

In the following, we prove $\textbf{(H1)}$-$\textbf{(H5)}$.

$\textbf{(H1)}$: Since $(1+\frac{1}{2m})^\mu>2$, we have
\begin{align*}
\frac{\rho_0}{2^{\nu}}([\log\frac{1}{\eta^{m+1}}]+1)^\mu&=\frac{\rho_0}{2^{\nu}}((1+\frac{1}{2m})^\nu\log\frac{1}{\eta_0^{m+1}}+1)^\mu\\
&\geq\frac{\rho_0}{2^{\nu}}2^{\nu}(\log\frac{1}{\eta_0^{m+1}})^\mu\\
&\geq1.
\end{align*}
It follows from the above that
\begin{align*}
&3\mu n\log([\log{\frac{1}{\eta^{m+1}}}]+1)-\frac{\rho_0}{2^\nu}([\log\frac{1}{\eta^{m+1}}]+1)^{3\mu}\\
&3\mu n\log([\log{\frac{1}{\eta^{m+1}}}]+1)-(\log\frac{1}{\eta^{m+1}})^{2\mu}\\
&\leq-\log\frac{1}{\eta^{m+1}},
\end{align*}
as $\eta$ is small enough, which is ensured by making $\varepsilon$ small.
Thus,
\begin{equation*}
K_{\nu+1}^{n}{\rm e}^{-K_{\nu+1}\rho_\nu}\leq \eta_\nu^{m+1},
\end{equation*}
i.e. $\textbf{(H1)}$ holds.

$\textbf{(H2)}$: By a similar method as \cite{chow}, we can prove \textbf{(H2)}.

It is easy to see that \textbf{(H3)} and \textbf{(H4)} hold by the below (\ref{<<1}).


$\textbf{(H5)}$:
We recall $\Delta$ here what has been defined  in \textbf{(H5)} and estimate it term by term.
In view of the definition of $A_\rho$ in (\ref{Arho}) and $\rho_\nu=\frac{\rho_0}{2^\nu}$, we calculate
\begin{align}\label{<<1}
\frac{A_\rho\eta_0^{\frac{1}{4}(1+\frac{1}{2m})^\nu}}{\rho_\nu}&=\frac{(2^n\sum_{0<|k|<K_+}|k|^{4\tau+2}{\rm e}^{\frac{-2|k|\rho_0}{2^\nu}})^\frac{1}{2}2^\nu\eta_0^{\frac{1}{4}(1+\frac{1}{2m})^\nu}}{\rho_0}\notag\\
&\leq\frac{(2^n(\frac{2^{\nu-1}}{\rho_0})^{4\tau+3}(4\tau+2)!)^\frac{1}{2}2^\nu\eta_0^{\frac{1}{4}(1+\frac{1}{2m})^\nu}}{\rho_0}\notag\\
&\leq\frac{(2^{n+2}(4\tau+2)!)^\frac{1}{2}(2^{2\tau+\frac{5}{2}})^{\nu-1}\eta_0^{\frac{1}{4}(1+\frac{1}{2m})^{\nu-1}}\eta_0^{\frac{1}{4}(1+\frac{1}{2m})}}{\rho_0^{2\tau+\frac{5}{2}}}\notag\\
&\leq\frac{(2^{n+2}(4\tau+2)!\eta_0^{\frac{1}{2}(1+\frac{1}{2m})})^\frac{1}{2}}{\rho_0^{2\tau+\frac{5}{2}}}(2^{2\tau+\frac{5}{2}})^{\nu-1}\eta_0^{\frac{1}{4}(1+\frac{1}{2m})^{\nu-1}}\notag\\
&\ll1.
\end{align}
Recall that
$\gamma_0=\varepsilon^{\frac{1}{4(2b)^{m+2}\Xi}}$, $\eta_0=\gamma_0^{2(2b)^{m+2}}\varepsilon^{\frac{1}{\Xi}}$, we have $$\eta_0<\gamma_0^{2(2b)^{m+2}}.$$ Observe that $\eta_\nu=\eta_0^{(1+\frac{1}{2m})^\nu}$, and $\gamma_\nu=\frac{\gamma_0}{2}(1+\frac{1}{2^\nu})$. Then $$\eta_\nu<\gamma_\nu^{2(2b)^{m+2}}$$
and
\begin{align}\label{eta+}
\eta_\nu^{1+\frac{1}{4}}<\gamma_{\nu+1}^{2(2b)^{m+2}}.
\end{align}
So \begin{align}\label{P+3}c\frac{A_{\rho_\nu}}{\rho_\nu}\eta_\nu^{\frac{3}{2}}\leq\gamma_{\nu+1}^{2(2b)^{m+2}}.\end{align}
Obviously, \begin{align}\label{P+2}c\eta_\nu^{\frac{1}{2}}\gamma_\nu^{(2b)^{m+2}}\leq\gamma_{\nu+1}^{2(2b)^{m+2}}.\end{align}
In view of the definition of $r_\nu$, $\eta_\nu$, we have  $$r_\nu^{m-2a}\eta_\nu^{-\frac{9}{4}}\leq \eta_0^{(m(m-2a)-\frac{9}{4})(1+\frac{1}{2m})^\nu}<1, as~ m>4,$$
this together with (\ref{<<1}) and (\ref{eta+}) yields \begin{align}\label{P+1}
\frac{A_{\rho_\nu}^2\eta_\nu^{\frac{1}{2}}}{\rho_\nu^2}r_\nu^{m-2a}\eta_\nu^{-1}\leq r_\nu^{m-2a}\eta_\nu^{-\frac{9}{4}}\eta_\nu^{1+\frac{1}{4}}\leq\gamma_{\nu+1}^{2(2b)^{m+2}}.
\end{align}

Combine (\ref{P+3}), (\ref{P+2}) with (\ref{P+1}), (\textbf{H5}) holds.

Above all, the KAM steps described in Section \ref{sec:KAM} are valid for all $\nu$, which gives the desired sequences stated in the lemma.

Let $\theta\gg1$ be fixed and $\eta_0$ be small enough so that
\begin{equation}\label{eta0eq38}
\eta_0<(\frac{1}{\theta})^{2m}<1.
\end{equation}
Then
\begin{align}
\eta_1&=\eta_0^{1+\frac{1}{2m}}<\frac{1}{\theta}\eta_0<1,\notag\\
\eta_2&=\eta_1^{1+\frac{1}{2m}}<\frac{1}{\theta}\eta_1<\frac{1}{\theta^2}\eta_0,\notag\\
\vdots\notag\\\label{etanueq39}
\eta_\nu&=\eta_{\nu-1}^{1+\frac{1}{2m}}<\cdots<\frac{1}{\theta^\nu}\eta_0.
\end{align}
Let $\theta\geq2$ in (\ref{etanueq39}). We have that for all $\nu\geq1$
\begin{align}\label{cetanu}
c_0\eta_\nu&\leq\frac{\eta_0}{2^\nu}\leq\frac{\eta_*^\frac{1}{2}}{2^\nu}.
\end{align}

Now, (\ref{eq30}), (\ref{eq36}) and (\ref{eq38}) follow from (\ref{e+}), (\ref{g+}), (\ref{f+}) and $\zeta_+\in B_{(r_-^{m-1}\eta_-^m)^{\frac{1}{L}}}(\zeta)$ in Lemma \ref{le3} and (\ref{cetanu}); by adding up (\ref{eq30}), (\ref{eq36}) and (\ref{eq38}) for all $\nu=0,1,\cdots$, we can get (\ref{eq31}), (\ref{eq37}) and (\ref{eq39}), respectively; (\ref{eq40}) follows from (\ref{XP+}) in Lemma \ref{le8} and (\ref{cetanu}); (\ref{eq41}) follows from $\zeta_+\in B_{(r_-^{m-1}\eta_-^m)^{\frac{1}{L}}}(\zeta)$ in Lemma \ref{le3} and (\ref{cetanu}); (\ref{eq32}) and (\ref{eq34}) follow from (\ref{omega+}), (\ref{Omega+}), (\ref{XR}), and (\ref{cetanu}); by adding up (\ref{eq32}) and (\ref{eq34}) for all $\nu=0,1,\cdots$, we can get (\ref{eq33}) and (\ref{eq35}).

$(2)$ follows from Lemma \ref{le7}.
Since $$\Pi_{\nu+1}=\Pi_\nu\setminus\cup_{K_{\nu-1}<|k|<K_\nu,|\ell|\leq2}\mathcal{R}_{kl}^{\nu+1}(\gamma_{\nu+1}),$$
we have
\begin{align}
|\Pi_{\nu}-\Pi_{\nu+1}|\leq\sum_{K_{\nu-1}<|k|\leq K_{\nu},~ |\ell|\leq2}|\mathcal{R}_{kl}^{\nu+1}(\gamma_{\nu+1})|\leq\sum_{K_{\nu-1}<|k|\leq K_{\nu},~ |\ell|\leq2}\frac{\gamma_0}{(|k|+1)^\tau}<c\gamma_1\frac{1}{1+K_{m-1}}.
\end{align}
The detailed proof can be seen in \cite{wu}.

In view of (\ref{omega+}) and (\ref{Omega+}), the Lipschitz semi-norm of the new frequency can be bounded as following
\begin{align*} |\omega_{\nu+1}(\xi)|_{\Pi_{\nu+1}}^\mathcal{L}+\pmb|\Omega_{\nu+1}(\xi)\pmb|_{-\delta,\Pi_{\nu+1}}^\mathcal{L}\leq M_{\nu}+c\pmb\|X_P\pmb\|_r^{\mathcal{L}}
\leq M_{\nu+1}.
\end{align*}

The proof is complete.
\end{proof}

\subsection{Convergence}
The convergence is standard, see \cite{poschel2}. For the sake of completeness, we briefly give the framework of proof.
Let
\begin{align*}
\Psi^\nu=\Phi_1\circ\Phi_2\circ\cdots\circ\Phi_\nu,~~~~\nu=1,2,\cdots.
\end{align*}
By Lemma \ref{le9}, we have
\begin{align*}
D_{\nu+1}&\subset D_\nu,\\
\Psi^\nu&:{D}_\nu\rightarrow {D}_0,\\
H_0\circ\Psi^\nu&=H_\nu=N_\nu+P_\nu,~~~\nu=0,1,\cdots,
\end{align*}
where $\Psi_0=id$.

To prove the convergence of the sequence $\Psi^\nu$, we note that the operator norm $\pmb\|\cdot\pmb\|_{r,s}$ satisfies
\begin{align*}
\pmb\|AB\pmb\|_{r,s}\leq \pmb\|A\pmb\|_{r,r}\pmb\|B\pmb\|_{s,s}, ~~~r\geq s.
\end{align*}
By the mean value theorem, we thus obtain
\begin{align}\label{psinu+1}
\pmb\|\Psi^{\nu+1}-\Psi^\nu\pmb\|_{r_0,D_{\nu+1}}\leq\pmb\|D\Psi^\nu\pmb\|_{r_0,r_\nu,D_\nu}\pmb\|\Phi_{\nu+1}-id\pmb\|_{r_\nu,D_{\nu+1}}.
\end{align}
In view of the chain rule $D\Psi^\nu=D\Phi_1\circ\cdots\circ D\Phi_\nu$ and (\ref{Dphi-I}), we have
\begin{align}\label{dpsi}
\pmb\|D\Psi^\nu\pmb\|_{r_0,r_\nu,D_\nu}\leq\prod_{i=0}^{\nu}\pmb\|D\Phi_i\pmb\|_{r_i,r_i,D_i}\leq\prod_{i\geq0}(1+\frac{A_\rho}{\rho_i}r_i^{m-2a}\eta_i^{m-a})\leq2,
\end{align}
for all $\nu\geq0$.

Using (\ref{eq36}), (\ref{psinu+1}), (\ref{dpsi}) and the identity
\begin{align*}
\Psi^\nu=id+\sum_{i=1}^\nu(\Psi^i-\Psi^{i-1}),
\end{align*}
it is easy to verify that $\Psi^\nu$ is uniformly convergent and denote the limitation by $\Psi^*$.

By Lemma \ref{le9},we see that $e_\nu$, $ \omega_\nu$, $\Omega_\nu$, $ g_\nu$, $f_\nu$ and $\zeta_\nu$ are uniformly convergent and denote the limitation by  $e_*$, $ \omega_*$, $\Omega_*$, $g_*$, $f_*$ and $\zeta_*$, respectively. It follows from Lemma \ref{le3} that $$\nabla g_*(0)=\cdots=\nabla g_\nu(0)=\cdots=\nabla g_0(0)=0.$$
Then, $N_\nu$ converge uniformly to
\begin{align*}
N_*=e_*(\xi)+\langle\omega_*(\xi), y\rangle+\langle w, \Omega_*(\xi)\bar{w}\rangle+g_*(z,\xi)+f_*(y, z, w, \bar w),
\end{align*}
with
\begin{align*}
g_*(z,\xi)&=g(z,\xi)+\varepsilon^{\frac{3m}{32\mu(m+1)(m-a)(\tau+1)}}O(||z||_{\textsf{a},\textsf{p}}^2),\\
f_*(y,z,w,\bar w,\xi)&=\sum_{4\leq2|\imath|\leq m}f_{\imath000} y^{\imath}+\sum_{2|\imath|+|\jmath|\leq m,1\leq|\imath|,|\jmath|}f_{\imath\jmath00} y^{\imath} z^{\jmath}+\sum_{0<2|\imath|+|\jmath|\leq m}f_{\imath\jmath11} y^{\imath} z^{\jmath} w\bar w.
\end{align*}
Hence \begin{align*}
P_\nu=H_0\circ\Psi^\nu-N_\nu
\end{align*}
converges uniformly to
\begin{equation*}
P_*=H_0\circ\Psi^*-N_*.
\end{equation*}

On the embedded tori, the flow of the perturbed Hamiltonian $H$ can be computed as follows. Note that
\begin{align*}
\pmb\|X_H\circ\Psi^\nu-D\Psi^\nu\cdot X_{N_\nu}\pmb\|\leq\pmb\|D\Psi^\nu\pmb\|_{r_0,r_\nu,D_\nu}\pmb\|(\Psi^\nu)^*X_H-X_{N_\nu}\pmb\|_{r_\nu,D_\nu}\leq c\pmb\|X_{P_\nu}\pmb\|_{r_\nu,D_\nu},
\end{align*}
whence in the limit, $X_H\circ\Psi^*=D\Psi^*\cdot X_{N^*}$.

Thanks to Wu and Yuan \cite{wu}, we can complete the measure estimate, 
we thus omit the details here.

\section{Example}\label{sec:example}

In this section, we introduce an example to demonstrate the existence of quasi-periodic solution by Theorem~1.

Consider the following Hamiltonian lattice,
\begin{eqnarray}\label{HL}
H&=&\sum_{j\in{{\mathbb Z}^+}\setminus \Lambda}\frac{\alpha^2_j}{2}q^2_{j}+\frac{1}{2}p^2_j
+\sum_{j\in\Lambda}\beta_{j}W(q_j,p_j)+\e \sum_{j\in{\mathbb Z}^{+}} V(q_{j+1}-q_j),
\end{eqnarray}
where $q_j,~p_j\in {\mathbb R}$. For fixed positive integer $n_1$ and $n_2$, the set $\Lambda$
 is defined as  $\Lambda:=\{j|~n_1< j\le n_2,~j\in{\mathbb Z}^+\}$.
 Moreover, let $\alpha:=(\alpha_1,\cdots,\alpha_{n_1})^{\top}$ be the vector-parameter varying in certain closed region
 ${\cal O}\subset{\mathbb R}^{n_1}$, $\beta_{j}$, $j\in\Lambda$ and $\alpha_j$,~$j>n_2$ be fixed constants.
   The functions are defined as
   \begin{eqnarray*}
  && W(q_j,p_j):=q_j^4-6q_j^2p_j^2+p^4_j, \quad j\in\Lambda\\
   && V(q_j,p_j):=\frac{1}{\alpha+1}(q_{j+1}-q_j)^{\alpha+1},\quad j\in{\mathbb Z}^+,
   \end{eqnarray*}
    for fixed $\alpha>0$.

When $\Lambda$ is empty, Hamiltonian (\ref{HL}) can be seen as the energy function of the following
Newton's cradle lattices with as small Hertzian interaction, that is,
\begin{eqnarray}\label{newtoncradle}
\ddot{q}_j+\alpha^2_jq_j=\e{\rm D}V(q_{j+1}-q_j)-\e{\rm D}V(q_j-q_{j-1}).
\end{eqnarray}
Newton's cradle lattice is known as a simplified model for granular chains consisting of
linear pendular and nonlinear interaction in form of Hertz's forces. The existence of
(quasi-)periodic breather and the corresponding stability were studied in \cite{geng,GJ,GJ2}.
The methods mentioned in these works are both numerical simulations and theoretical proof,
such as KAM and Nash-Moser iterations. We mention that, in the view of infinite dimensional
Hamiltonian normal form, the Hamiltonian lattice with respect to equation (\ref{newtoncradle})
are non-degenerate in normal direction. In contract of that, Hamiltonian (\ref{HL}) allows
degeneracy in the direction of $(q_j,p_j)$ for $j\in\Lambda$.

Introduce the standard action-angle form variables $(x,y)$,
where $x\in{\mathbb T}^{n_1}$,~$y\in{\mathbb R}^{n_1}$ and the normal variables  $u,\bar{u}\in{l}^{\textsf{a},\textsf{p}}$.
We also denote as above that $u=(w_0,w)^{\top}$ and $\bar{u}=(\bar{w}_0,\bar{w})$,
where $w_0,~\bar{w}_0\in{\mathbb R}^{n_2-n_1}$ and $w,\bar{w}$ are vectors in infinite dimension.
The transformations are as follows,
\begin{eqnarray*}
&&q_j:=\sqrt{\frac{2}{\alpha_j}}\sqrt{y_j}\cos\theta_j,\quad\quad\quad p_{j}:=\sqrt{2\alpha_j}\sqrt{y_j}\sin\theta_j,\quad\quad\quad 1\le j\le n_1,\\
&&q_{j}:=\frac{w_{0,j-n_1}+\bar{w}_{0,j-n_1}}{2},\quad\quad~~~ p_{j}:=\frac{\bar{w}_{0,j-n_1}-\bar{w}_{0,j-n_1}}{2\sqrt{-1}},\quad\quad~~~ n_1+1\le j\le n_2,\\
&&q_{j}:=\frac{1}{\sqrt{2\alpha_j}}(w_{j-n_1}+\bar{w}_{j-n_1}),\quad
 p_{j}:=\sqrt{\frac{\alpha_j}{-2}}(\bar{w}_{j-n_1}-\bar{w}_{j-n_1}),\quad n_2+1\le j.
\end{eqnarray*}
 Then,
Hamiltonian (\ref{HL}) can be reduced into the following form, that is,
\begin{eqnarray}\label{reducedHL}
H(x,y,u,\bar{u},\alpha):=\langle\omega(\alpha),y\rangle+\langle w,\Omega\bar{w}\rangle+g(w_0,\bar{w}_0)+\e P(x,y,u,\bar{u},\alpha),
\end{eqnarray}
where
\begin{align*}
&\omega(\alpha):=\alpha,\\
&\Omega:={\rm diag}\{\alpha_{n_2+1},\alpha_{n_2+2},\cdots,\alpha_{n_2+j},\cdots\},\\
&g(w_0,\bar w_0):=\sum_{j=n_1+1}^{n_2}\beta_{j}(\frac{|w_{0,j-n_1}|^4}{2}+\frac{|\bar{w}_{0,j-n_1|}|^4}{2}),
\end{align*}
and the perturbation reads as
\begin{eqnarray*}
P(x,y,u,\bar{u},\alpha)&:=&\frac{y_1}{\alpha_1}\cos\theta_1+\sum^{n_1}_{j=2}\frac{2y_j}{\alpha_j}\cos\theta_j
-\sqrt{\frac{y_{n_1}}{2\alpha_{n_1}}}\cos\theta_{n_1}(w_{0,1}+\bar{w}_{0,1})\\
&~&-\sum_{j=1}^{n_1-1}2\sqrt{\frac{y_{j+1}y_j}{\alpha_{j+1}\alpha_{j}}}\cos\theta_{j+1}\cos\theta_{j}+O(|(u,\bar{u})|^2),
\end{eqnarray*}
by choosing $\alpha=1$ in (\ref{HL}).

It is obvious that $g(w_0,\bar{w}_0)$ are in the same form as it is mentioned in Remark~\ref{remark1} by choosing $p=q=2$,  so that it satisfies
assumption (\textbf{A0}). Moreover, choosing a fixed point $y_*=(y_{*1},~\cdots,~y_{*n_1})^{\top}$ with $y_{*j}\ne 0$,
the perturbation $P$ is real analytic with respect to $(x,y,u,\bar{u})$ in the following
complex neighborhood, that is,
$$
D(s,r)=\{(x,y,u,\bar{u}):~|{\rm Im}~x|<s,~\|y-y_*\|<r^2,~\|(w_0,\bar{w}_0)\|<r,~\|(w,\bar{w})\|<r^a\},
$$
where $0<s,r\ll1$,~$a\ge2$ is defined as in \ref{a}. It is obviously that (\textbf{A3}) is satisfied.
Then we obtain the following result.
\begin{corollary}
Consider Hamiltonian lattice (\ref{HL}), as well as the reduced system (\ref{reducedHL}), assume that
the tangent and normal frequencies $\omega(\alpha)$ and $\Omega$ satisfy assumptions \textbf{\textsc{(A1)}}-\textbf{\textsc{(A2)}}. It follows
from Theorem 1 that Hamiltonian lattice (\ref{HL}) admits a family of real analytic
embedding of a $n_1$-dimensional tori for the majority of $\alpha\in{\cal O}$.
\end{corollary}
\begin{remark}
We have to mention that, as a demonstration of Theorem 1, we simply choose $\alpha=1$.
However, $\alpha$ may not be an integer in certain applications. It leads to the result
that the perturbation is not real analytic with respect to $(u,~\bar{u})$ near the
origin so that KAM iteration in present paper can not be directly applied. Hence, we look
forward to generalize the result to such a system which  not only allows degeneracy in normal direction but also  $C^{1+\alpha}$
with respect to normal variables.
\end{remark}

\section{Appendix A. \textbf{Proof of Proposition \ref{example}}}\label{pro}
\begin{proof}
Obviously, for $\forall z\in (-2,2)\times(-2,2)$,
$$\nabla g(-z)=-\nabla g(z),~~~~\nabla g(0)=0,$$
and for $\forall z\in \partial(-2,2)\times(-2,2)$, $$\nabla g(z)\neq0,$$
Using Borsuk's theorem, we have
$$\deg(\nabla g(z),(-2,2)\times(-2,2),0)\neq0,$$
i.e., the topological degree condition in \textbf{(A0)} holds. For $\forall z, z_*\in[-1,1],$ and $z\neq z_*$, we have $$\nabla g(z)-\nabla g(z_*)=0,$$
but $$|z-z_*|^L>0,~~~\forall L\geq2,$$
which shows that the weak convexity condition in \textbf{(A0)} fails. Note that the perturbed motion equation in the direction of $w_0$ is $$\dot{w_0}=\bar w_0+\varepsilon^\ell\sin\frac{1}{\varepsilon}.$$
In order to ensure the existence of low dimensional invariant tori, we need to solve the following equation $$\bar w_0+\varepsilon^\ell\sin\frac{1}{\varepsilon}=0,$$
which implies that $\bar w_0$ is discontinuous and alternately appears on $(-2,-1)$ and $(1,2)$ as $\varepsilon\rightarrow0_+$. So, this example shows that the weak convexity condition in \textbf{(A0)} is necessary.

\end{proof}

\section*{Acknowledgments}

The second author is supported by  National Natural Science Foundation of China (grant No.12271204), Project of Science and Technology Development of Jilin Province, China (grant No.20200201265JC). The third author is supported by National Basic Research Program of China (grant No. 2013CB834100), National Natural Science Foundation of China (grant No. 11571065, 11171132, 12071175), Project of Science and Technology Development of Jilin Province, China (grant No. 2017C0281, 20190201302JC),
and Natural Science Foundation of Jilin Province (grant No. 20200201253JC).


\end{document}